\documentclass{elsarticle}
\usepackage{imakeidx}
\usepackage{enumitem}
\makeindex[intoc]

\newcommand{\IN}[1]{\index{#1|BH}#1}
\usepackage[utf8]{inputenc}
\usepackage[margin=1in]{geometry} 
\usepackage[toc,page]{appendix}
\usepackage{amsmath,amsthm,amssymb}
\usepackage{bbm}
\usepackage{dsfont}
\usepackage{setspace}
\usepackage{multicol}
\usepackage{xcolor}
\usepackage{graphicx}
\usepackage{tikz}
\usetikzlibrary{shapes.geometric, arrows}
\usetikzlibrary{arrows.meta}
\usepackage{capt-of}
\usepackage[colorlinks]{hyperref}
\hypersetup{
    colorlinks=true,
    linkcolor=blue,
    filecolor=magenta,      
    urlcolor=cyan,
    citecolor=red
}
\usepackage{cleveref}
\usepackage{titleref}
\allowdisplaybreaks

\newcommand{\teddydone}[1]{{\color{black} #1}}
\newcommand{\mahya}[1]{{\color{black} #1}}
\usepackage[normalem]{ulem}
\newcommand{\stkout}[1]{\ifmmode\text{\sout{\ensuremath{#1}}}\else\sout{#1}\fi}

\def\R{{\mathbb R}} 
\def\Z{{\mathbb Z}} 
\def\N{{\mathbb N}} 
\def\G{{\mathcal G}} 
\def\W{{\mathcal W}} 
\def\B{{\mathcal B}}
\def\V{{\mathcal V}}
\def\UL{{\text{UL}}}
\def\LR{{\text{LR}}}
\def\I{{\mathcal I}}

\newtheorem*{lemma*}{Lemma}

\theoremstyle{plain}
\newtheorem{theorem}{Theorem}[section]
\newtheorem{proposition}[theorem]{Proposition}
\newtheorem{lemma}[theorem]{Lemma}
\newtheorem{claim}[theorem]{Claim}
\newtheorem{corollary}[theorem]{Corollary}

\theoremstyle{definition}
\newtheorem{definition}[theorem]{Definition}

\newtheorem{assumption}[theorem]{Assumption}
\theoremstyle{remark}
\newtheorem{remark}[theorem]{Remark}

\numberwithin{subcase}{case}
\newtheorem{notation}[theorem]{Notation}

\begin{document}
\begin{frontmatter}

\title{Robust Recovery of Robinson Property in $L^p$-Graphons: A Cut-Norm Approach}

\address[add1]{Department of Mathematical Sciences, Ewing Hall, University of Delaware in Newark, DE}
\address[add2]{Department of Mathematics, Victoria Building, Toronto Metropolitan University in Toronto, ON}

\author[add1]{Mahya Ghandehari}
\ead{mahya@udel.edu}

\author[add2]{Teddy Mishura}
\ead{tmishura@torontomu.ca}

\begin{abstract}
This paper investigates the Robinson graphon completion/recovery problem within the class of $L^p$-graphons, focusing on the range $5<p\leq \infty$. 
A graphon $w$ is Robinson if it satisfies the Robinson property: if $x\leq y\leq z$, then $w(x,z)\leq \min\{w(x,y),w(y,z)\}$. 
We demonstrate that if a graphon possesses localized near-Robinson characteristics, it can be effectively approximated by a Robinson graphon in terms of cut-norm. 
To achieve this recovery result, we introduce a function $\Lambda$, defined  on the space of $L^p$-graphons, 
which quantifies the degree to which a graphon $w$ adheres to the Robinson property. 
We prove that $\Lambda$ is a suitable gauge for measuring the Robinson property when proximity of graphons is understood in terms of cut-norm. Namely, we show that (1) $\Lambda(w)=0$ precisely when $w$ is Robinson; (2) $\Lambda$ is cut-norm continuous, in the sense that if two graphons are close in the cut-norm, then their $\Lambda$ values are close; and (3) for $p > 5$, any $L^p$-graphon $w$ can be approximated by a Robinson graphon, with error of the approximation bounded in terms of $\Lambda(w)$. When viewing $w$ as a noisy version of a Robinson graphon, our method provides a concrete recipe for recovering a cut-norm approximation of a noiseless $w$. 
Given that any symmetric matrix is a special type of graphon, our results can be applicable to symmetric matrices of any size.
Our work extends and improves previous results, where a similar question for the special case of $L^\infty$-graphons was answered.
\end{abstract}

\end{frontmatter}
\tableofcontents
\section{Introduction}
Given an incomplete matrix $M_I$ with many missing entries,  is it possible to recover the underlying complete matrix $M$?
An instance of such a scenario is when data is aggregated from customer ratings, where any given customer only rates a few items. In general, this problem is ill-posed; indeed, with low dimensional information, it is not possible to uniquely recover high dimensional data. However, in practical applications, it is often the case that additional structure about the original matrix $M$ is known. In the above example regarding costumer ratings, it is usually assumed that most people's preferences are based on very few factors, implying that the complete matrix $M$ should be low rank. Leveraging this assumption makes the completion of $M_I$ quite tractable, even if the majority of its entries are missing. 
In general, the problem of filling in the incomplete entries in a matrix $M_I$ under the condition that the completed matrix satisfies a specified structural property is known as the \emph{matrix completion problem}. 
Due to its many applications in data science, the completion problem for low rank matrices has gained the attention of many researchers (see e.g.~\cite{Candes2009, Tao2010,   Gross, Jain, keshavan2010, recht2011, SunLuo,  Vandereycken}).
Furthermore, the matrix completion problem has been explored in the context of various other essential structural properties, including Hankel, Toeplitz, and moment structures (see \cite{eftekhari2018, Fazel,GOYENS2023498,WANG2016133} for some examples).
Unfortunately, even known information can be corrupted by noise. 
In this case, the problem of completion becomes that of \emph{recovery}: Given an observed matrix $\hat{M}$, find a matrix $M$ satisfying a prescribed structural property such that $\|M-\hat{M}\|$ is small.  
For examples of recovery of noisy matrices, see 
\cite{agarwal2012,candes2010noise,Keshavan-noisy,klopp-noisy,Klopp2017,Koltchinskii2011,Rohde,tao2011} for low rank matrices, \cite{flammarion2019,ma2021} for monotone matrices, and \cite{cai2011,chen2015,rudelson2013} for covariance matrices.

%
An important structural property of matrices/graphons is the \textit{Robinson} property.
A \emph{Robinson matrix}, also called an \emph{R-matrix}, is a symmetric matrix $A=[a_{ij}]$ such that for $i \leq  j \leq k$ we have 
\begin{equation*}
    a_{ik} \leq \min\{a_{ij},a_{jk}\}.
\end{equation*}  
Robinson matrices are well-studied objects in data science, largely because of their connection with the  classical seriation problem \cite{fogel2014,Liiv_2010,Mirkin1984GraphsAG,prea-fortin,seston}. The objective of the seriation problem is to use pairwise comparisons of a set of items to recover their linear ordering. 
The seriation problem is easily transferred to the problem of determining whether a given symmetric matrix can be permuted into a Robinson matrix, and finding the Robinson ordering. This problem can be solved in polynomial time; see \cite{Mirkin1984GraphsAG} for the original algorithm, and \cite{atkins1998,Fortin2017RobinsonianMR,LAURENT2017151,laurent2017} for more recent efficient algorithms. 
When the underlying matrix is perturbed by noise, recovery of a Robinson ordering turns out to be a very challenging problem. 
For a discussion on the NP-hardness of this problem when  $\ell^p$-norm approximations (with $p<\infty$) are required, see \cite{barthelemy2001np}; and for a polynomial time algorithm when seeking $\ell^\infty$-norm approximations, see \cite{chepoi2009}.

To put matrices of various size in one framework, and also to develop a probabilistic Robinson theory that can treat matrices of growing sizes, we focus on \emph{graphons} as a versatile extension of symmetric matrices. 
Graphons are measurable real-valued functions on $[0,1]^2$ that are symmetric with respect to the diagonal. A graphon is called an $L^p$-graphon if its $p$-norm is finite. 
$L^{\infty}$-graphons are typically referred to as graphons, and were introduced in \cite{LOVASZ2006933} as the limit objects of converging sequences of dense graphs (i.e.~graph sequences $\{G_n\}_{n\in{\mathbb N}}$ in which the number of edges is quadratic in the number of vertices).
%
In \cite{Borgs_2018,Borgs_2019}, Borgs {et al.}~extended this point of view, and proved that $L^p$-graphons (when $1 < p < \infty$)  are limit objects of sparse graph sequences.
The mode of convergence in dense/sparse graph limit theory is easily captured using the \textit{cut-norm}. This norm was introduced in \cite{friezekannan} for matrices, and can be extended to graphons in a natural manner: 
\begin{equation*}
\|w\|_{\square}=\sup_{A,B \subseteq [0,1]}\Bigg|\iint_{A\times B} w(x,y)\, dxdy\Bigg|,
\end{equation*}
where the supremum is taken over all  Lebesgue measurable subsets $A,B$.

Graph limit theory provides an effective approach for capturing the common large-scale features of networks.
This viewpoint is now the base for many network analysis methods.
In the  limit theory of graphs, {graphons} represent random processes, called $w$-random graphs, that generate networks. Networks produced by the same graphon have similar large-scale features.
For this reason, we think of graphons as the large-scale blueprint of any graph that they generate.
Naturally, to study networks, we only need to study the (underlying) graphon, rather than the individual networks generated by that graphon. 
We refer to \cite{lovaszbook} for a comprehensive account of (dense) graph limit theory. To see some applications of graphons in network analysis, we refer to \cite{uncertainty,cgao2015,cgao2021,sgao2020,Ghanehari-Janssen-Kalyaniwalla,PDE1,Medvedev2014}.

Robinson $L^p$-graphons are defined similarly to Robinson matrices (see Subsection~\ref{subsec-robinson}). 
In this paper, we study the graphon recovery problem for the class of Robinson graphons. 
The limiting Robinson graphon of a converging graph sequence, if recovered successfully in terms of cut-norm, can be used to seriate the graphs in that sequence (see for example \cite{janssen2022}  for a randomized algorithm or \cite{Aaron-opr} for a spectral approach).
For this reason, we must measure proximity in terms of cut-norm when performing the Robinson graphon recovery.
Our objective, therefore, is to obtain a statement of the following general form:
\begin{quote}
\emph{Given a graphon $w$ that is locally almost Robinson, there exists a Robinson graphon $u$ such that $\|w-u\|_{\Box}$ is sufficiently small.}
\end{quote}
To begin, it is essential to formalize the concept of locally almost Robinson graphons. 
In this paper, we introduce a graphon function denoted as $\Lambda$, which serves as a gauge for assessing the Robinson property by quantifying localized (aggregated) deviations from it  (Definition~\ref{def:lambda}). 
In informal terms, we consider a graphon $w$ to be locally almost Robinson if $\Lambda(w)$ is close to 0.
Indeed, we prove that $\Lambda$ serves as a suitable measurement for the Robinson property when $p > 5$  (see Proposition~\ref{prop:Lambda-prop} and Theorem~\ref{thm:main-result}); that is, we show that $\Lambda$ is subadditive and satisfies the following three properties:
\begin{itemize}
    \item\textbf{(Recognition)} $\Lambda(w) = 0$ if and only if $w$ is Robinson.

    \item\textbf{(Continuity)} $\Lambda$ is continuous with respect to the cut-norm.

    \item\textbf{(Recovery)} Given an $L^p$-graphon $w$, there exists a Robinson graphon $u$ such that $\|w-u\|_{\Box} \leq c_p\Lambda(w)^{\alpha_p}$, 
    where $\alpha_p\in (0,1)$ and $c_p>0$ are constants depending only on $p$.
\end{itemize}
The final item in the list presented above may be interpreted as follows: if $w$ is locally almost Robinson (i.e., $\Lambda(w)$ is near 0), then a Robinson graphon $u$ can be found so that $\|w-u\|_{\Box}$ is also small. We think of $u$ as a \emph{Robinson approximation} for $w$.
Our results (in particular, Definition~\ref{def:R(w)} and Theorem~\ref{thm:main-result}) provide a systematic approach for creating Robinson approximations for graphons that are almost Robinson on a local scale.

The current work extends the results of \cite{Chuangpishit_2015,ghandehari2020graph}, where a similar problem was addressed within the context of $L^\infty$-graphons. Specifically, \cite{Chuangpishit_2015} introduced a function $\Gamma$ defined on the space of $L^\infty$-graphons. Much like $\Lambda$, this function $\Gamma$ adheres to three crucial properties: 
(1) recognition: $\Gamma(w)=0$ precisely when $w$ is Robinson, (2) continuity: when equipped with the cut-norm, the function $\Gamma$ exhibits continuity over the space of $L^\infty$-graphons, and (3) recovery: for every $L^\infty$-graphon $w$, there exists a Robinson $L^\infty$-graphon $u$ such that $\|w-u\|_{\Box} \leq 14\Lambda(w)^{1/7}$.
The approach in the current paper offers a two-fold advantage. Firstly, our newly proposed function $\Lambda$ possesses the desired properties of recognition, continuity, and recovery for all $L^p$-graphons with $p>5$, extending beyond the restricted case of $L^\infty$-graphons. 
Secondly, the formula defining the function $\Lambda$ is less intricate compared to $\Gamma$.
Leveraging the simplicity of $\Lambda$ and applying meticulous estimates, we not only extend the results achieved by $\Gamma$ to encompass $L^p$-graphons (as detailed in Theorem~\ref{thm:main-result}), but also notably improve the outcomes within the domain of $L^{\infty}$-graphons (as demonstrated in Corollary~\ref{cor:bnddstability}).

It is worth noting that the proof of continuity and recovery of $\Gamma$ in \cite{Chuangpishit_2015,ghandehari2020graph}  cannot be adjusted to work for the setting of $L^p$-graphons. 
\mahya{Indeed, the proof of the $\|\cdot\|_\Box$-continuity of $\Gamma$ in \cite{Chuangpishit_2015}  relies upon the continuity of the triangular cut operator when acting on the space of $[-1,1]$-valued graphons \cite[Lemma 6.1]{Chuangpishit_2015}. This continuity, however, does not hold when the triangular cut is applied to $L^p$-graphons with $p<\infty$ (see \cite{MISHURA202226}). As a result, the proof in \cite{Chuangpishit_2015} do not work for the unbounded case. 
Similarly, the proof of stability of $\Gamma$ in \cite[Theorem 3.2]{ghandehari2020graph} heavily relies on the fact that the graphons at hand do not attain values outside of $[-1,1]$ (see Remark~\ref{remark:teddy} for more details).  
} 
Thus, to  handle $L^p$-graphons, the new approach presented here seemed to be necessary. 
%

This paper is organized as follows: In Section \ref{sec:notationbg}, we introduce various necessary background involving graphons and the study of the Robinson property. Section \ref{sec:previousparam} discusses previous work that has been done to recover Robinson graphons in the cut-norm; here, we also comparatively summarize our own graphon recovery results. In Section \ref{sec:lambda}, alongside the building of necessary machinery and statement of technical lemmas, we present proofs of our main recovery results. The aforementioned technical lemmas are proven in Appendix \ref{section:proofs of lemmas}. 

\section{Notation and background}\label{sec:notationbg}
In this section of the paper, we present necessary background and notation for the reader. We begin with basic notation used throughout the paper; afterwards, we present graph limit theory, stating required definitions alongside notable results, then transition into explanation of the Robinson property for both matrices and graphons. 

All functions and sets discussed in this paper are assumed to be measurable. The symmetric difference between two sets $A$ and $B$ is denoted $A \triangle B$, and the measure of $A$ is denoted $|A|$. The notation $\|\cdot\|_p$ is used to represent the standard $p$-norm; that is, for a function $f:X \to \R$ on a measure space $X$, and $1 \leq p < \infty$,
\begin{equation*}
    \|f\|_p := \left(\int_X |f|^p\right)^{\frac{1}{p}} \mbox{ and } \|f\|_{\infty} := \inf_{a \in \R}\left\{|f(x)| \leq a \mbox{ almost everywhere (a.e.)}\right\}.
\end{equation*}
\subsection{Graphons and the cut-norm}
We begin by defining the graphon, the fundamental building block of graph limit theory, in several of its myriad forms.
\begin{definition}[Graphon space]
\begin{itemize}
    \item[]
    \item Let $\W_0$ be the set of all symmetric, measurable functions $w:[0,1]^2 \to [0,1]$. The elements of this set are called \textit{graphons}.

    \item Let $\W^{\infty}$ be the set of all bounded, symmetric, measurable functions $w:[0,1]^2 \to \R$. The elements of this set are called \textit{kernels}.

    \item For $p\geq 1$, let $\W^p$ be the set of all symmetric, measurable functions $w:[0,1]^2 \to \R$ such that $\|w\|_p < \infty$. The elements of this set are called $L^p$-\textit{graphons}.
\end{itemize}
\end{definition}
It is clear that $\W_0 \subseteq \W^{\infty} \subseteq \W^p \subseteq \W^1$ for $p > 1$. We mention as well the space of \emph{step kernels} $\mathcal{S} \subseteq \W^{\infty}$, which comprises the set of symmetric step functions on $[0,1]^2$; that is, for $w \in \mathcal{S}$, there exists some partition $\mathcal{P}= \{P_1,\ldots,P_m\}$ of $[0,1]$ into measurable subsets such that $w$ is constant on $P_i \times P_j$ for all $1 \leq i,j \leq m$. We define the space of step graphons $\mathcal{S}_0 \subseteq \W_0$ similarly. The norm of choice for these function spaces is the \textit{cut-norm}, which was introduced in \cite{friezekannan} for matrices, and can  be naturally extended to the case of graphons as follows.
\begin{definition}[Cut-norm]
Let $w \in \W^1$. We define the \textit{cut-norm} by
\begin{equation*}
    \|w\|_{\square} = \sup_{S,T \subseteq [0,1]}\Bigg|\iint_{S \times T}w(x,y)\, dxdy\Bigg|
\end{equation*}
where the supremum is taken over all measurable subsets $S$ and $T$. Moreover, the supremum in the formula for cut-norm is achieved \cite[Lemma 8.10]{lovaszbook}.
\end{definition}
\noindent It is obvious that for any function $w \in \W^1$, and for all $p \geq 1$,
\begin{equation*}
    \|w\|_{\square} \leq \|w\|_{1} \leq \|w\|_{p} \leq \|w\|_{\infty}.
\end{equation*}
Similar to unlabeled graphs, we are interested in `unlabeled graphons'; these are defined through an analytic generalization of graph vertex relabelling.
Let $\Phi$ denote the set of all measure-preserving bijections on $[0,1]$; that is, if $A \subset [0,1]$ and $\phi\in \Phi$, then $|\phi(A)|=|\phi^{-1}(A)| = |A|$. The \textit{cut-distance} between two graphons $u$ and $w$ is
\begin{equation*}
    \delta_{\square}(u,w) = \inf_{\phi \in \Phi}\|u-w^{\phi}\|_{\square},
\end{equation*}
where $w^{\phi}(x,y) = w(\phi(x),\phi(y)).$ The cut-distance $\delta_{\square}$ is only a pseudometric; we thus identify graphons of cut distance 0 to get the space $\widetilde{\W}_0$ of \textit{unlabeled graphons}. $\widetilde{\W}$ and $\widetilde{\W}^p$ are defined similarly. This has quite the reward: For $p > 1$, the metric space $(\widetilde{W}^p,\delta_{\Box})$ has a compact unit ball (\cite[Theorem 2.13]{Borgs_2019}). For $p=1$, additional conditions are required to ensure compactness (\cite[Theorem C7]{Borgs_2019}), though we do not consider this case in the paper.

Let $w \in \W^1$ and let $A,B \subseteq [0,1]$. The \textit{cell average} of $w$ over $A \times B$ is given by 
\begin{equation*}
\overline{w}(A\times B) = \frac{1}{|A \times B|}\iint_{A\times B}w~dxdy.
\end{equation*}
For $w \in \W^1$ and a partition $\mathcal{P} = (S_1,...,S_k)$ of $[0,1]$ into measurable sets, we define the function $w_{\mathcal{P}}$ by
\begin{equation*}
    w_{\mathcal{P}}(s,t) = \frac{1}{|S_i\times S_j|}\iint_{S_i \times S_j}w(x,y)\, dxdy = \overline{w}(S_i\times S_j)\quad \text{if } (s,t) \in S_i\times S_j,
\end{equation*}
The operator $w \mapsto w_{\mathcal{P}}$ is called the \textit{stepping operator}. It is easy to see that the stepping operator is contractive with respect to the cut-norm as well as $L^p$-norms ; indeed, for $w \in \W^1$, a partition $\mathcal{P}$ of $[0,1]$, and $p \geq 1$, we have
\begin{equation*}
    \|w_{\mathcal{P}}\|_p \leq \|w\|_p ~\text{ and } ~\|w_{\mathcal{P}}\|_{\square} \leq \|w\|_{\square}.
\end{equation*}

\subsection{The Robinson property}\label{subsec-robinson}
The Robinson property for matrices was first introduced in \cite{robinson_1951} for the study of the classical seriation problem, whose objective is to order a set of items so that similar items are placed close to one another. A symmetric matrix $A=[a_{ij}]$ is a \emph{Robinson matrix} if
\begin{equation*}
   i \leq j \leq k \implies a_{ik} \leq \min\{a_{ij},a_{jk}\}.
\end{equation*}  
and is \textit{Robinsonian} if it becomes a Robinson matrix after simultaneous application of a permutation $\pi$ to its rows and columns. The permutation $\pi$ is a \textit{Robinson ordering} of $A$. If the entries $a_{ij}$ of the symmetric matrix $A$ represent similarity of items $i$ and $j$, then the Robinson ordering of $A$ represents a linear arrangement of the items so that similar items are placed closer together. 

The Robinson property for graphons, initially introduced in \cite{Chuangpishit_2015}, shares a similar definition. 
When analyzing large matrices or graphs, particularly in the context of network visualization, it is useful to determine whether these complex structures can be interpreted as samples of a Robinson graphon $w$. 
If so, one can harness the Robinson property inherent in the graphon $w$ to obtain quantitative insights concerning the data at hand. 
Formally, an $L^p$-graphon $w \in \W^p$ is said to be \emph{Robinson} if 
\begin{equation*}
   x \leq y \leq z \implies w(x,z) \leq \min\{w(x,y),w(y,z)\}.
\end{equation*}  
We call an $L^p$-graphon \textit{Robinson almost everywhere}, or Robinson a.e.~for short, if it is equal a.e.~to a Robinson $L^p$-graphon. 
An  $L^p$-graphon $w \in \W^p$ is called \textit{Robinsonian} if there exists a Robinson $L^p$-graphon $u \in \W^p$ such that $\delta_{\Box}(w,u) = 0$. 

\section{Robinson property: measurement and recovery}\label{sec:previousparam}
In order to study the recovery problem for Robinson matrices/graphons, we need to devise a graphon parameter that suitably measures the Robinson property. Such a parameter must \emph{recognize} Robinson graphons, be \emph{continuous} on the space of graphons, and \emph{recover} the Robinson property. 
These three key features ensure that the graphon parameter recognizes graphs sampled from Robinson graphons as well.
Informally speaking, these are graphs whose parameter value is small, and we think of them as `almost Robinson graphs'. 

We  give a brief overview of past results from \cite{Chuangpishit_2015} introducing such a parameter for $L^\infty$-graphons, before discussing efforts to tackle this problem in the world of $L^p$-graphons.
The parameter $\Gamma:\W_0 \to [0,1]$ was introduced in \cite{Chuangpishit_2015} to estimate to what extent a given graphon fails to satisfy the Robinson property. We recall the definition of \IN{$\Gamma$}, noting that it can naturally extend  to $\W^1$. For $w \in \W^1$ and a measurable subset $A$ of $[0,1]$, we define $\Gamma(w,A)$ as
\begin{equation*}
    \Gamma(w,A) =\iint_{y<z}\left[\int_{x \in A \cap [0,y]} (w(x,z)-w(x,y))dx\right]_+dydz+ \iint_{y<z}\left[\int_{x \in A \cap [z,1]} (w(x,y)-w(x,z))dx\right]_+dydz,
\end{equation*}
where $[x]_+ = \max(x,0)$. We then define $\Gamma(w) := \sup_{A}\Gamma(w,A)$, where the supremum is taken over all  measurable subsets of $[0,1]$. It turns out that $\Gamma$ suitably measures the Robinson property for $L^\infty$-graphons:
\begin{itemize}
 \item[(i)] \label{item:gamma_one}({\bf Recognition} \cite[Proposition 4.2]{Chuangpishit_2015}) $w \in \W_0$ is Robinson a.e. if and only if $\Gamma(w)=0$.
 \item[(ii)] \label{item:gamma_two} ({\bf Continuity} \cite[Lemma 6.2]{Chuangpishit_2015}) $\Gamma$ is continuous on $\W_0$ with respect to cut-norm.
 \item[(iii)] \label{item:gamma_three}{(\bf Recovery} \cite[Theorem 3.2]{ghandehari2020graph})
For every $w \in \W_0$, there exists a Robinson graphon $u \in \W_0$ satisfying
\begin{equation*}
    \|u-w\|_{\Box} \leq 14\Gamma(w)^{1/7}.
\end{equation*} 
\end{itemize}
Indeed, $\Gamma$ recognizes samples of Robinson graphons: These are precisely graph sequences whose $\Gamma$-values converge to 0 \cite[Theorem 1.2]{ghandehari2020graph}.
\begin{remark}\label{remark:teddy}
\mahya{It is natural to ask whether these results for $\Gamma$ can be extended to $L^p$-graphons; 
however, in trying to do so, one encounters key issues in proving continuity \hyperref[item:gamma_two]{(ii)} and recovery \hyperref[item:gamma_three]{(iii)} of $\Gamma$.
Firstly, if $\chi$ denotes the characteristic function of the triangle above the diagonal, then the proof of \hyperref[item:gamma_two]{(ii)} is based on the following two facts:
\begin{itemize}
    \item[(a)]\label{item:gamma_a} $|\Gamma(w_1)-\Gamma(w_2)|\leq 2\|w_1-w_2\|_{\Box} + 2\|(w_1-w_2)\chi\|_{\Box}$ for $w_1,w_2 \in \W^1$.
    \item[(b)]\label{item:gamma_b} $\|w\chi\|_\Box\leq 2\sqrt{\|w\|_\Box\|w\|_\infty}$ for $w \in \W^{\infty}.$
\end{itemize}
While condition \hyperref[item:gamma_a]{(a)} holds in the sparse case, it was shown in \cite{MISHURA202226} that condition \hyperref[item:gamma_b]{(b)} cannot be generalized to $w\in \mathcal{W}^p$ with $1< p<\infty$. 
In fact, the map $M_{\chi}: w \mapsto w\chi$ is not $\|\cdot\|_{\Box}$-continuous on $\W^p$ for $1 < p < \infty$ \cite[Corollary 3.4]{MISHURA202226}. So even though the definition of \IN{$\Gamma$} can be extended to \IN{$L^p$-graphon}s, 
the proof of continuity in \cite{Chuangpishit_2015} does not naturally extend to the unbounded case. Whether $\Gamma$ is cut-norm continuous on $\W^p$ remains an open question.}

\mahya{Secondly, the proof of stability of $\Gamma$ in \cite{Ghanehari-Janssen-Kalyaniwalla} relies heavily, and at multiple points, on the universal $L^\infty$-bound on graphons. 
Most importantly, the key idea of the proof is to locate a large enough cell on which the graphon has a large enough average (\cite[Claim 4.11]{Ghanehari-Janssen-Kalyaniwalla}), assuming that the desired upper bound on the error of estimate fails. The existence of such a cell, which is an integral part of the proof, is not guaranteed when the graphon is not $L^\infty$-bounded.
Indeed, for any $1\leq p <\infty$, it is easy to construct a graphon $w$ with $\|w\|_p=1$ for which such a cell does not exist.}
\end{remark}
%
\section{Robinson recovery of $L^p$-graphons}\label{sec:lambda}
In this section, we introduce a new graphon parameter $\Lambda$ for measuring the Robinson property. We say that $A \leq B$ for sets $A,B \subseteq[0,1]$ if $a \leq b$ for all $a \in A$ and $b \in B$. 
\begin{definition}[Robinson parameter]\label{def:lambda}
    Let $w \in \mathcal{W}^1$. Define
\begin{equation}\label{eq:lambda}
    \Lambda(w) = \frac{1}{2}\sup_{\substack{A \leq B \leq C, \\ |A|=|B|=|C|}}\bigg[\iint_{A\times C}w~dxdy-\iint_{B\times C} w~dxdy\bigg]\nonumber\\
    + \frac{1}{2}\sup_{\substack{X \leq Y \leq Z, \\ |X|=|Y|=|Z|}}\bigg[\iint_{X\times Z}w~dxdy-\iint_{X\times Y} w~dxdy\bigg]
\end{equation}
where $A,B,C$ and $X,Y,Z$ are measurable subsets of $[0,1]$. 
\end{definition}
It is clear that we have $\Lambda(w) \geq 0$ for all  $w\in \W^1$, as we can take $A=B=C=X=Y=Z=\emptyset$.  Moreover, 
\begin{equation*}
\Lambda(w)\leq \frac{1}{2}\Bigg(\iint_{(A\cup B)\times C}|w|\, dxdy +\iint_{X\times (Y\cup Z)}|w|\, dxdy\Bigg),
\end{equation*}
implying that $\Lambda(w) \leq \|w\|_p$ for every $p\geq 1$. 
Note that since the supremum is subadditive, we also have that $\Lambda$ is subadditive, i.e., for all $u,w \in \W^p$
\begin{equation}\label{eq:Lambda-prop}
    \Lambda(u+w) \leq \Lambda(u) + \Lambda(w).
\end{equation}
The rest of this section is dedicated to proving that $\Lambda$ suitably measures the Robinson property for $L^p$-graphons. We thus show that $\Lambda$ recognizes the Robinson property, is continuous on $L^p$-graphons, and can recover the Robinson property. We begin with the proof of recognition, noting that for ease of notation, the upper triangle of the unit square will be written
\begin{equation*}
\Delta=\left\{(x,y)\in[0,1]^2:\ x\leq y\right\}.
\end{equation*}
\begin{proposition}[Recognition]\label{prop:Lambda-prop}
Let $1\leq p\leq \infty$ and suppose $w \in \W^p$. Then $w$ is Robinson a.e.~if and only if $\Lambda(w)=0$.
\end{proposition}
\begin{proof}
Observe that if $w$ is Robinson a.e., then $\Lambda(w) = 0$. To prove the reverse direction, suppose $\Lambda(w) = 0$. For $n \in \mathbb{N}$, let $w_n = w_{\mathcal{P}_n}$, where $\mathcal{P}_n$ is the partition of $[0,1]$ into $n$ equal-size intervals $I_1\leq I_2\leq \ldots \leq I_n$. Note that as $\Lambda(w) = 0,$ for any choice of measurable subsets $A \leq B \leq C$ of $[0,1]$ of equal size, we have%
\begin{equation*}
    \iint_{A \times C} w(x,y)\, dxdy \leq \iint_{B \times C}w(x,y)\, dxdy \ \mbox{ and }\  \iint_{A \times C} w(x,y)\, dxdy \leq \iint_{A \times B}w(x,y)\, dxdy.
\end{equation*}
Applying the above inequalities to the sets $I_i$, we observe that $w_n$ is Robinson. Since $w_n\to w$ in the $L^1$ norm, by using the Borel--Cantelli Lemma and going down to a subsequence if necessary, we can assume that $\{w_n\}$ converges to $w$ pointwise a.e.~in $[0,1]^2$. So, there exists a set of measure zero $N \subset [0,1]^2$ such that  $w(x,y) = \lim_{n}w_n(x,y)$ for $(x,y) \in [0,1]^2\setminus N$. As $w$ is symmetric, we can assume that $N$ is symmetric with respect to the main diagonal. Next, we define the symmetric function $\widetilde{w}$ as follows:
\begin{equation*}
    \widetilde{w}(x,y) = \begin{cases}
    w(x,y) & (x,y) \in \Delta \setminus N\\
    \sup\left\{w(u,v):\ (u,v) \in ([0,x]\times[y,1])\setminus N\right\} & (x,y) \in \Delta\cap N
    \end{cases}.
\end{equation*}
Clearly $w = \widetilde{w}$ almost everywhere. To show $\widetilde{w}$ is Robinson, we observe that for every $(x_1,y_1), (x_2,y_2)\in \Delta$ with $x_2\leq x_1$ and $y_1\leq y_2$, we  have $\widetilde{w}(x_2,y_2) \leq \widetilde{w}(x_1,y_1)$. Indeed, the required inequality can be easily verified by considering four cases depending on whether each of $(x_1,y_1)$ and $(x_2,y_2)$ belong to $\Delta\setminus N$ or $\Delta\cap N$; we omit the details of this straightforward verification.
\end{proof}
We now show that $\Lambda$ is continuous on $L^p$-graphons.
\begin{proposition}[Continuity]
$|\Lambda(w)-\Lambda(u)|\leq 2\|w-u\|_{\Box}$. \label{item:Lambda-cts}
\end{proposition}
\begin{proof}
Let $w,u \in \mathcal{W}^1$, and fix $\epsilon > 0$. There exist measurable sets $A\leq B\leq C$ with equal size and measurable sets $X\leq Y\leq Z$ with equal size such that
\begin{equation}\label{eq:lambdaepbound}
\frac{1}{2}\bigg(\iint_{A\times C}w~dxdy-\iint_{B\times C}w~dxdy + \iint_{X\times Z}w~dxdy-\iint_{X\times Y} w~dxdy\bigg) \geq \Lambda(w)-\epsilon.
\end{equation}
Combining \eqref{eq:lambdaepbound} with the definition of $\Lambda(u)$, we get that
\begin{eqnarray*}
    2(\Lambda(w)-\Lambda(u)-\epsilon) \leq
    \iint_{A \times C}(w-u) + \iint_{B \times C}(u-w) 
  + \iint_{X\times Z} (w-u)+\iint_{X \times Y}(u-w)
    \leq 4\|w-u\|_{\square}.
\end{eqnarray*}
Following the same logic, we can also prove that $\Lambda(u)-\Lambda(w)\leq 2\|w-u\|_{\square}+2\epsilon$. Taking $\epsilon \to 0$, we get $|\Lambda(w)-\Lambda(u)| \leq 2\|w-u\|_{\square}.$ 
\end{proof}
Having shown that $\Lambda$ recognizes the Robinson property and is continuous on $L^p$-graphons, all that remains is to show it recovers the Robinson property; that is, we endeavour to show that for a given $L^p$-graphon $w$, there exists some Robinson $L^p$-graphon $u$ such that $\|w-u\|_{\Box}$ is bounded above by some power of $\Lambda(w)$. This is a far more challenging task than the previous two properties and must be handled in several parts. We begin by first defining this Robinson graphon $u$ and listing some properties that will be useful in future proofs.
\subsection{Robinson approximation of $L^p$-graphons}\label{subse:region}
\begin{figure}[t]
    \centering
    \begin{minipage}{0.4\textwidth}
    \includegraphics[width=\textwidth]{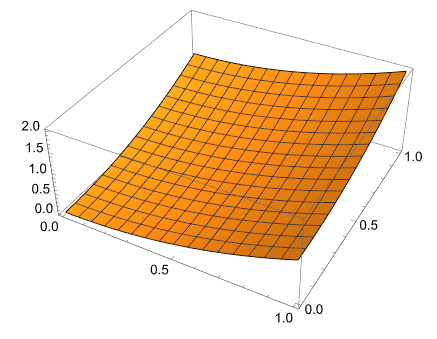}
    \end{minipage}
    \begin{minipage}{0.4\textwidth}
    \includegraphics[width=\textwidth]{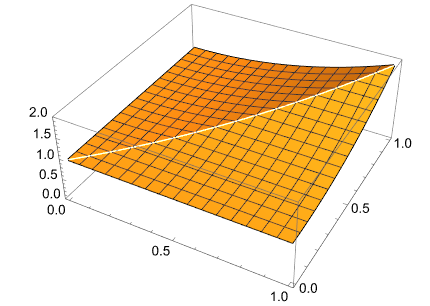}
    \end{minipage}
     \caption{\mahya{The graphon $w(x,y)=x^2+y^2$ (left) and its Robinson approximation $R_w^{\alpha}$ for $\alpha =0.05$ (right).}}
     \label{fig:robinson_example}
\end{figure}
Here we introduce the $\alpha$-Robinson approximation of an $L^p$-graphon. We borrow the following terminology from \cite{ghandehari2020graph}. The \textit{upper left (UL)} and \textit{lower right (LR)} regions of a given point $(a,b) \in \Delta$ are given by
\begin{align*}
\UL(a,b)&=[0,a]\times[b,1],\\
\LR(a,b)&=[a,b]\times[a,b]\cap \Delta.
\end{align*}
\begin{definition}[$\alpha$-Robinson approximation for graphons]\label{def:R(w)} 
Let $p \geq 1$ and fix a parameter $0<\alpha<1$. Given $w\in {\mathcal W}^p$, the $\alpha$-\emph{Robinson approximation} $R_w^{\alpha}$ of $w$ is the $L^p$-graphon such that for all $(x,y)\in \Delta$,
\begin{equation}
R_w^{\alpha}(x,y)=\sup\left\{ \overline{w}(S\times T)\,:\, S\times T\subseteq \UL(x,y),\, |S|=|T|=\alpha\right\},
\end{equation}
where $S,T$ are measurable and $\sup\emptyset =0$. Moreover, we set $R_w^{\alpha}=w$ if $\alpha=0$ and $w$ is Robinson.
\end{definition}
\mahya{For every $w\in \W^p$, the  approximation  $R_w^{\alpha}$ is a Robinson graphon. To see this, note that for two points $(x_1,y_1)$ and $(x_2,y_2)$, if $(x_1,y_1) \in \UL(x_2,y_2)$, then $\UL(x_1,y_1) \subseteq \UL(x_2,y_2)$.
So by definition of $R_w^{\alpha}$, we get
$R_w^{\alpha}(x_1,y_1) \leq R_w^{\alpha}(x_2,y_2)$.
Therefore, $R_w^{\alpha}$ satisfies the Robinson property (Subsection~\ref{subsec-robinson}), as 
for $x\leq y\leq z$, we have $(x,z)\in \UL(x,y)\cap \UL(y,z)$.
See Figure \ref{fig:robinson_example} for an example of a graphon $w$ and its Robinson approximation $R_w^{\alpha}$.}
\begin{remark}
If $w$ is a Robinson graphon and $0<\alpha<1$, then the graphon $R_w^\alpha$ can  be different from $w$, although both graphons are Robinson.
This is indeed the case for Robinson step graphons. In general, $w$ does not need to be continuous, but $R_w^\alpha$ satisfies certain continuity properties: Though $R_w^{\alpha}$ is certainly not always continuous on $[0,1]^2$--the outer boundary of thickness $\alpha$ of the unit square is set to 0 by definition--it is continuous on $[\alpha,1-\alpha]^2$. We prove this fact and some other useful properties of $R_w^{\alpha}$ in Proposition \ref{prop:rw-properties}.
\end{remark}
\begin{proposition}[Properties of the Robinson approximation]\label{prop:rw-properties}
Let $p \geq 1$, and fix a parameter $0<\alpha<1$.
Given graphons $w,u\in {\mathcal W}^p$  , we have:
\begin{enumerate}
    \item[(i)] If  $u \leq w$ pointwise, then $0 \leq R_w^{\alpha}-R_u^{\alpha} \leq R_{w-u}^{\alpha}.$
    
    \item[(ii)] $R_w^{\alpha}$ is continuous on the square $[\alpha,1-\alpha]^2$, and $\|R_w^{\alpha}\|_{\infty} \leq \alpha^{-\frac{2}{p}}\|w\|_p$.

    \item[(iii)] For $w,u\in \W^p$, we have $\|R_w^\alpha-R_u^\alpha\|_\infty\leq \|w-u\|_\Box$. In particular,
    $\|R_w^\alpha-R_u^\alpha\|_\Box\leq \|w-u\|_\Box$.
\end{enumerate}
We note that for item $(iii)$, only convergence in the $L^1$-norm---the weakest $L^p$-norm---is required. Therefore, convergence in any $L^p$-norm is enough to guarantee the convergence of the $\alpha$-Robinson approximations in the cut-norm.
\end{proposition}
\begin{proof}
To prove (i), let $(x,y) \in \Delta$. By definition of the $\alpha$-Robinson approximation, for every $\epsilon > 0$, there exist sets $A,B \subseteq \UL(x,y)$ such that $R_w^{\alpha}(x,y) \geq \overline{w}(A\times B) \geq R_w^{\alpha}(x,y)-\epsilon$ and $C,D \subseteq \UL(x,y)$ such that $R_u^{\alpha}(x,y)\geq\overline{u}(C\times D) \geq R_u^{\alpha}(x,y)-\epsilon$. We also note that $|A|=|B|=|C|=|D|=\alpha$. Since $u\leq w$, we have
\begin{equation*}
    0 \leq \frac{1}{\alpha^2}\iint_{C\times D} (w - u)dxdy
   =\frac{1}{\alpha^2}\iint_{C\times D} w~dxdy-\frac{1}{\alpha^2}\iint_{C\times D} u~dxdy
   \leq R_w^{\alpha}(x,y)-R_{u}^{\alpha}(x,y)+\epsilon
\end{equation*}
Letting $\epsilon \to 0$, we conclude that $0 \leq R_w^{\alpha}-R_u^{\alpha}$. Moreover, we note that
\begin{equation*}
   R_w^{\alpha}(x,y)-\epsilon-R_{u}^{\alpha}(x,y)  \leq \frac{1}{\alpha^2}\iint_{A\times B} w~dxdy - \frac{1}{\alpha^2}\iint_{A \times B} u~dxdy 
= \frac{1}{\alpha^2}\iint_{A\times B} (w - u)dxdy 
\leq R_{w-u}^{\alpha}(x,y),
\end{equation*}
showing that (i) holds true, again by letting $\epsilon \to 0$. 

To prove (ii), let $\{(x_n,y_n)\} \subset [\alpha,1-\alpha]^2$ be a sequence converging to $(x,y) \in [\alpha,1-\alpha]^2$ in the standard Euclidean distance. Then, we have 
$\UL(\min(x_n,x),\max(y_n,y)) \subseteq \UL(x,y) \subseteq \UL(\max(x_n,x),\min(y_n,y)),$
which implies that
\begin{equation}\label{eq:rwbounds}
    R_w^{\alpha}(\min(x_n,x),\max(y_n,y)) \leq R_w^{\alpha}(x_n,y_n) \leq R_w^{\alpha}(\max(x_n,x),\min(y_n,y)).
\end{equation}
We shall show that both the upper and lower bound in \eqref{eq:rwbounds} converge to $R_w^{\alpha}(x,y)$; by the Squeeze Theorem, this will imply our original claim. To do so, we let $\epsilon > 0$ and recall from elementary measure theory that as $w$ is an $L^1$ function, for every $\epsilon > 0$, there exists a $\delta > 0$ such that for every $S \subset [\alpha,1-\alpha]^2$ satisfying $|S| < \delta$, we have $\iint_S |w|\, dxdy < \epsilon.$ Let $N \in \N$ be such that for all $n \geq N$, we have that $d((x_n,y_n),(x,y)) < \frac{\delta}{2}$. Defining the set $S := \UL(\min(x_n,x),\max(y_n,y))\triangle\UL(x,y)$, it must then be that 
\begin{equation*}
    R_w^{\alpha}(x,y) \leq R_w^{\alpha}(\min(x_n,x),\max(y_n,y))+\frac{1}{\alpha^2}\iint_{S}|w|dxdy \leq R_w^{\alpha}(\min(x_n,x),\max(y_n,y))+\frac{\epsilon}{\alpha^2}.
\end{equation*}
Note that by definition of $R_w^\alpha$, we have $R_w^{\alpha}(\min(x_n,x),\max(y_n,y)) \leq R_w^{\alpha}(x,y)$, so
\begin{equation*}
    |R_w^{\alpha}(\min(x_n,x),\max(y_n,y))-R_w^{\alpha}(x,y)| < \epsilon.
\end{equation*}
Let $\epsilon \to 0$, we conclude that the left hand side of \eqref{eq:rwbounds} converges to $R_w^{\alpha}(x,y).$ A similar argument shows that the the right hand side of \eqref{eq:rwbounds} converges to $R_w^{\alpha}(x,y)$ as well, and we are done.

To prove the norm bound, we note that for every $\epsilon > 0$, there exist sets $A,B \subseteq \UL(x,y)$ where $|A|=|B|=\alpha$ such that $R_w^{\alpha}(x,y) \geq \overline{w}(A\times B) \geq R_w^{\alpha}(x,y)-\epsilon$. Then, by H\"older's inequality,
\begin{equation*}
R_w^{\alpha}(x,y)-\epsilon \leq \frac{1}{\alpha^2}\iint_{A\times B} w~dxdy \leq \frac{1}{\alpha^2}\|w\|_p\|\mathbbm{1}_{A\times B}\|_q = \alpha^{-\frac{2}{p}}\|w\|_p,
\end{equation*}
where we used $\|\mathbbm{1}_{A\times B}\|_q=\alpha^{2/q}=\alpha^{2-2/p}.$ We let $\epsilon \to 0$ to finish the proof of (ii).

To prove (iii), fix $(x,y)\in \Delta$, and let $\epsilon>0$ be arbitrary. By an argument similar to (i), there exist sets $A,B$ of size $\alpha$ such that $A\times B\subseteq \UL(x,y)$, and 
$$R_w^{\alpha}(x,y)-R_u^\alpha(x,y)\leq \overline{w}(A\times B)-\overline{u}(A\times B)+\epsilon\leq \alpha^2\|w-u\|_\Box+\epsilon.$$
Sending  $\epsilon$ to 0, we get $R_w^{\alpha}(x,y)-R_u^\alpha(x,y)\leq \alpha^2\|w-u\|_\Box$. 
A similar argument can be used to show $R_u^{\alpha}(x,y)-R_w^\alpha(x,y)\leq \alpha^2\|w-u\|_\Box$, so we are done.
\end{proof}
%
\subsection{Recovering the Robinson property for $L^p$-graphons}
Our proof that $\Lambda$ recovers the Robinson property requires that the value of $R_{w}^{\alpha}$ and $\overline{w}$ both be tightly controlled within sets of specific measure. To accomplish this, we enhance the techniques outlined in \cite{ghandehari2020graph}, and as a result, we need to introduce analogous notation. Let $w \in \W^{\infty}$ be a  graphon with $w\geq 0 $ and $\lceil\|w\|_{\infty}\rceil=M$. Fix a parameter $\alpha\in (0,1)$ (as in Definition~\ref{def:R(w)}), and let $m$ be a fixed integer. For $k\in \{1,\ldots,mM-1\}$, define the $k$-th \emph{black region} $\B_k$, the $k$-th \emph{white region} $\V_k$ and the $k$-th \emph{grey region} $\G_k$ as follows. 
\begin{align*}
    \B_k&=\Big\{(x,y)\in \Delta:\ x=y\ \ \text{or} \ \  \exists \ S\times T\subseteq \UL(x,y) \  \text{with}\ |S|=|T|=\alpha \text{ and }   \overline{w}(S \times T)>\frac{k}{m}\Big\},\\
    \V_k&=\Big\{(x,y)\in \Delta \setminus {\B}_k:\ \exists \ S\times T\subseteq \LR(x,y)\  \text{with} \ |S|=|T|=\alpha \text{ and }  \overline{w}(S \times T)\leq \frac{k}{m}\Big\},\\
    \G_k&=\Delta\setminus (\B_k\cup\V_k).
\end{align*}
We set $\B_0=\V_{mM}=\Delta$, set $\V_0=\B_{mM}=\emptyset$, and denote $\G:=\bigcup_{k=1}^{mM-1} \G_k$. In addition, we also define the following regions:
\begin{align*}
    &\mathcal{R}_k:=\B_k\cap \V_{k+1}.
\end{align*}  
Black regions provide lower bounds on both $R_w^{\alpha}$ and $\overline{w}$ while white regions provide upper bounds. 
So the regions $\mathcal{R}_k$ are used to form tight bounds on the behaviour of both $R_w^{\alpha}$ and $\overline{w}$. 
Grey regions provide no information pertinent to our method of proof. \mahya{From the definition of the regions, it is easy to observe that 
$$\B_0\supseteq \B_1\supseteq\ldots\supseteq \B_{mM} \ \mbox{ and }\ \V_0\subseteq \V_1\subseteq\ldots\subseteq \V_{mM}.$$
Subsequently, we establish that the regions $\mathcal{R}_k$ form a partition of $\Delta\setminus \mathcal{G}$. Our approach mirrors that of \cite[Lemma 4.8]{ghandehari2020graph}, with a supplementary sketch of the proof provided for reader's convenience.} 
\begin{figure}
    \centering
    \begin{minipage}{0.4\textwidth}
    \includegraphics[width=\textwidth]{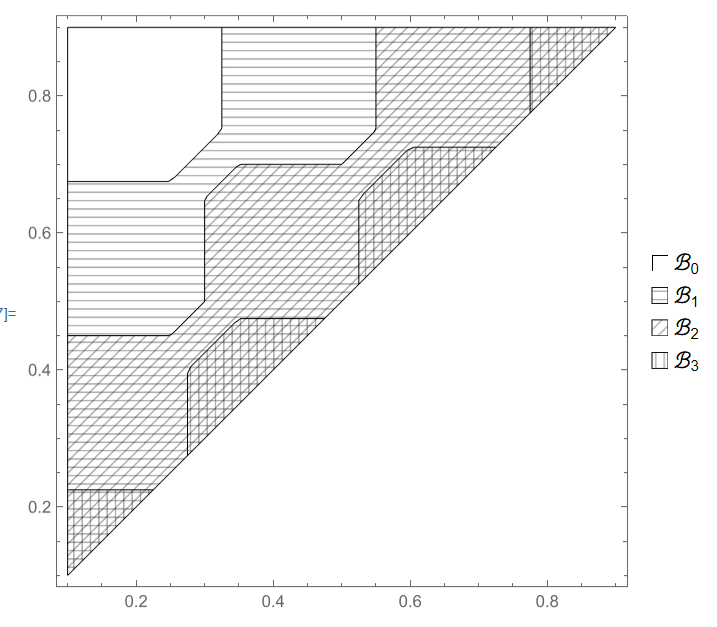}
    \end{minipage}
    \begin{minipage}{0.4\textwidth}
    \includegraphics[width=\textwidth]{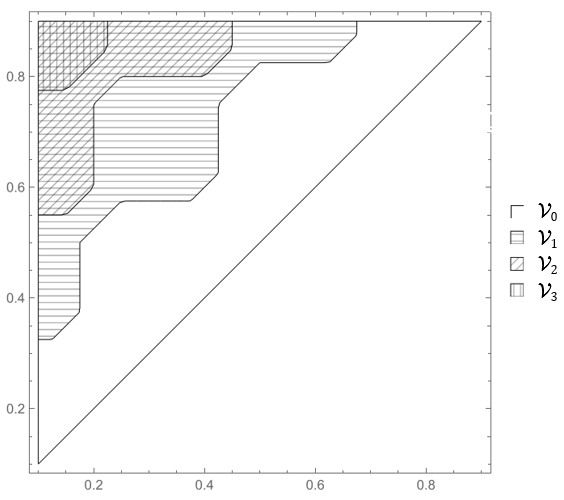}
    \end{minipage}
     \caption{An example of black and white regions for $m=4$ and $\alpha =0.1$.}
     \label{fig:bnwregions}
\end{figure}
\begin{lemma}\label{lem:rk-parts-delt/g}
Let $w \in {\mathcal W}^\infty$ be such that $w \geq 0$ and $\lceil \|w\|_{\infty} \rceil = M$, and let $\mathcal{B}_k$, $\mathcal{V}_k$, $\mathcal{G}_k$, \mahya{and $\mathcal{R}_k$} be as defined above. Then, 
$$\Delta \setminus \left(\bigcup_{k=1}^{mM-1}\mathcal{G}_k\right) = \mahya{\bigsqcup_{k=0}^{mM-1}}\mathcal{R}_k,$$
\mahya{where $\bigsqcup$ denotes the disjoint union of sets.}
\end{lemma}
\mahya{\begin{proof}
Let $1\leq i<j\leq mM-1$, and note that $\mathcal{R}_i\subseteq \V_{i+1}\subseteq \V_j$ and $\mathcal{R}_j\subseteq \B_{j}$. Since $\V_j\cap \B_j=\emptyset$, we conclude that $\mathcal{R}_i\cap \mathcal{R}_j=\emptyset$ as well. So the regions $\mathcal{R}_k$ are disjoint. Next, observe that
\begin{equation*}
    \Delta \setminus \left(\bigcup_{k=1}^{mM-1}\mathcal{G}_k\right) = \Delta \setminus \left(\bigcup_{k=1}^{mM-1}\Delta \setminus (\mathcal{B}_k \cup \mathcal{V}_k)\right) = \bigcap_{k=1}^{mM-1}(\mathcal{B}_k \cup \mathcal{V}_k).
\end{equation*}
Now, we consider the expansion of $(\mathcal{B}_1 \cup \mathcal{V}_1)\cap(\mathcal{B}_2 \cup \mathcal{V}_2)\cap\ldots\cap(\mathcal{B}_{mM-1} \cup \mathcal{V}_{mM-1})$ into expressions $X_1\cap\ldots\cap X_{mM-1}$ with $X_i \in \{\mathcal{B}_i,\mathcal{V}_i\}$. We further note that $X_1\cap\ldots\cap X_{mM-1}=\emptyset$ whenever $X_i = \mathcal{V}_i$ and $X_j = \mathcal{B}_j$ for some $i<j$; thus, every nonempty term $X_1\cap\ldots\cap X_{mM-1}$ from the above expansion must be of one of the following forms:
\begin{enumerate}
    \item[(i)] $X_1\cap\ldots\cap X_{mM-1} = \mathcal{B}_j \cap \mathcal{V}_{j+1} = \mathcal{R}_j$ with $1 \leq j < mM-1$ if there is at least one black and one white region amongst the $X_i$.
    
    \item[(ii)] $X_1\cap\ldots\cap X_{mM-1} = \mathcal{V}_1\cap\ldots\cap\mathcal{V}_{mM-1} = \mathcal{V}_1$ if all $X_i$ are white.
    
    \item[(iii)] $X_1\cap\ldots\cap X_{mM-1} = \mathcal{B}_1\cap\ldots\cap\mathcal{B}_{mM-1} = \mathcal{B}_{mM-1}$ if all $X_i$ are black.
\end{enumerate}
This completes the proof, as $\mathcal{V}_1 = \mathcal{V}_1\cap\mathcal{B}_0 = \mathcal{R}_0$ and $\mathcal{B}_{mM-1}=\mathcal{B}_{mM-1}\cap\mathcal{V}_{mM} = \mathcal{R}_{mM-1}$.
\end{proof}}

By definition, for every $k\in \{1,\ldots, mM-1\}$, the regions $\B_k$ and $\V_k$ have \emph{upper} and \emph{lower boundary} functions $f_k,g_k:[0,1]\rightarrow [0,1]$, where $f_k$ is the upper boundary of $\B_k$ and $g_k$ is the lower boundary of $\V_{k}$. These are defined in \cite{ghandehari2020graph} as follows: 
\begin{eqnarray*}
f_k(x) &=& \sup\{z \in [x,1]: (x,z) \in \mathcal{B}_k\}, \\
g_k(x) &=& \inf\{z \in [x,1]: (x,z) \in \mathcal{V}_k\};
\end{eqnarray*}
with the convention $\inf \emptyset = 1$. Additionally, we define $f_0(x) = 1$ and $g_{mM}(x) = x$ for all $x \in [0, 1]$ to represent the corresponding boundaries for $\mathcal{B}_0 = \mathcal{V}_{mM} = \Delta$. Finally, since $f_k$ and $g_{k+1}$ are the upper and lower boundaries of $\mathcal{B}_k$ and $\mathcal{V}_{k+1}$ respectively, if the region $\mathcal{R}_k$ is nonempty, then it is bounded from below by $g_{k+1}$ and from above by $f_k$. We refer to Figure \ref{fig:bnwregions} for a visual representation of the regions and \IN{boundary function}s previously mentioned.

From the definition of the black and white regions, it is easy to see that the functions $f_k$ and $g_k$ are both increasing functions and thus only admit jump discontinuities. We naturally extend the \IN{graph} of these functions to \textit{\IN{boundary curve}s} by adding vertical line segments connecting any such discontinuities, denoting the resulting curves once again by $f_k$ and $g_k$ respectively.

Let $S,T \subseteq [0,1]$ be measurable. We say that  $S \times T$ \textit{crosses} a \IN{boundary curve} $f_k$ or $g_k$ if the top-left corner of the cell is strictly above the \IN{boundary curve} and its bottom-right corner is strictly below the curve. This definition is used as $S \times T$ need not be a connected subset of $\R^2$, so a \IN{boundary curve} can go through the cell without intersecting it. Figure \ref{fig:crossing} depicts such behavior.
\begin{figure}[h]
    \centering
    \includegraphics[scale=0.4]{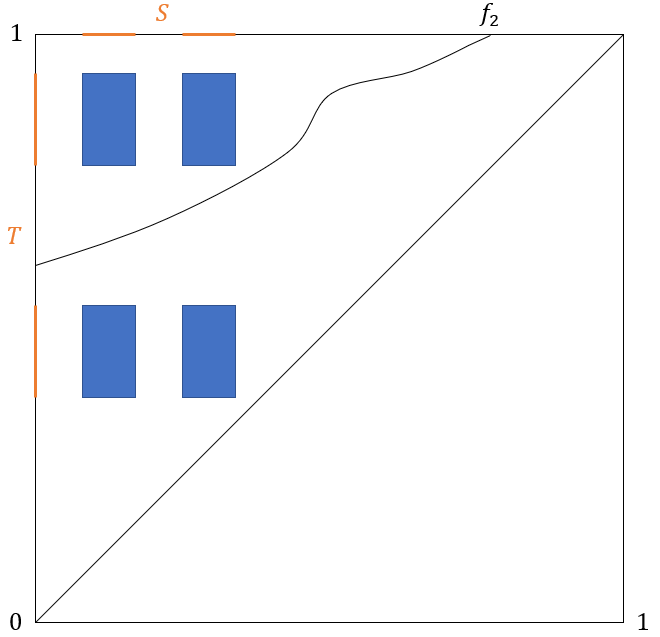}
    \caption{An example of a cell $S \times T$ that crosses a \IN{boundary curve} $f_2$ without intersecting it.}
    \label{fig:crossing}
\end{figure}
%

The proof of our recovery theorem is based on the idea that the total area of the grey regions $\mathcal{G}_k$ is small. This allows us to concentrate on the behaviour of $w$ and $R_w^{\alpha}$ inside the regions $\mathcal{R}_k$, where their values are strictly controlled. We show that their local average difference in these regions can be controlled by $\Lambda(w)$, leading to the conclusion that $\|w-R_w^{\alpha}\|_\Box$ must be small. We introduce three lemmas here, each necessary in the proof of our main result. We will present the proofs of Lemma~\ref{lem:splitint} and Lemma~\ref{lem:gnrlpigeon} in Appendix \ref{section:proofs of lemmas}. The proof of Lemma~\ref{lem:grey-regions-small} is very similar to \cite[Lemma 4.10]{ghandehari2020graph}, \mahya{but to keep this article self-contained, we present a sketch of this proof in Appendix \ref{section:proofs of lemmas} as well}.
\begin{lemma}\label{lem:grey-regions-small}
Let $k\in{\mathbb Z}^{\geq 0}$, $w \in \W^p$ with $w\geq 0$, and $\alpha \in (0,1)$. Then, $\overline{\G_k}$ does not contain any $\beta\times \beta$ square, where $\beta>\alpha$. Here, $\overline{\G_k}$ denotes the closure of $\G_k$ in the Euclidean topology of ${\mathbb R}^2$. 
\end{lemma}
\begin{lemma}\label{lem:splitint}
Let $u\in L^\infty[0,1]$ (not necessarily non-negative), and $P\subseteq [0,1]$ be a measurable subset such that $\int_P u~dx\neq 0$. Let $0 < \beta < |P|$ be fixed. Then $P$ can be partitioned into $N:=\lceil |P|/\beta\rceil$ subsets $P_1,\ldots,P_{N}$ so that the following conditions are satisfied:

\begin{onehalfspacing}
\begin{itemize}
    \item[(i)] $P_1 \leq \ldots \leq P_{N-1}.$
    
    \item[(ii)] $|P_i|= \beta$ for $1\leq i\leq N-1$ and $|P_{N}|\leq \beta$.
    
    \item[(iii)] $\big|\int_{P_{N}} u~dx\big|\leq \frac{1}{N}\big|\int_{P}u~dx\big|$. 
\end{itemize}
\end{onehalfspacing}
\end{lemma}
\begin{lemma}\label{lem:gnrlpigeon}
Let $f \in L^1[0,1]^2$, and let $S,S' \subseteq [0,1]$ be measurable subsets such that $|S|=|S'|$. Suppose for a constant $C > 0$ we have
\begin{equation*}
    \iint_{S \times S'} f~dxdy\geq C.
\end{equation*}
Then, for every $\alpha \in (0,1)$, there exist measurable sets $T\subset S$ and $T'\subset S'$ such that $|T| = |T'|=\alpha|S|$ and
\begin{equation*}
    \frac{1}{|T\times T'|}\iint_{T \times T'}f~dxdy \geq \frac{C}{|S \times S'|}.
\end{equation*}
\end{lemma}

\subsection{Proof of recovering property of $\Lambda$}
Prefacing our main result, we begin first with a necessary proposition---as the proof of Theorem~\ref{thm:main-result} heavily features ``cutting'' graphons off at certain values, we must control their behaviour once cut. The following proposition provides such control.

\begin{proposition}\label{prop:robinson-upprbnd}
Let $p>2$ and $w \in \W^{\infty}$ with $w\geq 0$. Suppose $\|w\|_p \leq 1$. If $R_w^{\alpha}$ is the \IN{Robinson approximation} of $w$ with parameter $\alpha=\|w\|_{\infty}^{-\frac{p}{3p-2}}\Lambda(w)^{\frac{2p}{5p-2}}$, then 
\begin{equation}\label{eq:approx-result}
\|w-R_w^{\alpha}\|_{\Box} \leq  \teddydone{4}\bigg(1+(8\|w-R_w^{\alpha}\|_p+2)\|w\|_{\infty}^{\frac{2p}{3p-2}}\bigg)\Lambda(w)^{\frac{p-2}{5p-2}}.
\end{equation}
\end{proposition}
The proof of this proposition is inspired by \cite[Theorem 3.2]{ghandehari2020graph}, 
however, the approximation result of \cite{ghandehari2020graph} would not provide the upper bound needed for Proposition~\ref{prop:robinson-upprbnd}.
As a corollary of this proposition, we improve upon the bound obtained in \cite{ghandehari2020graph}.
\begin{proof}
By \cite[Lemma 8.10]{lovaszbook}, there exist measurable $S,T\subseteq [0,1]$ so that
$$\left|\iint_{S\times T} (w-R_w^{\alpha})\, dxdy\right|= \| w-R_w^{\alpha}\|_{\Box}.$$
Replacing $S\times T$ with $T\times S$ if necessary, we can assume without loss of generality that 
\begin{equation}\label{eq;choice-D}
\bigg|\iint_{(S\times T)\cap \Delta} (w-R_w^{\alpha})\, dxdy\bigg| \, \geq\frac{1}{2}\| w-R_w^{\alpha}\|_{\Box}. 
\end{equation} 
Fix $\beta \in (\alpha, \teddydone{\frac{8}{7}}\alpha)$. Next, we split $S$ into $N_1:=\lceil |S|/\beta\rceil$ subsets $S_1, S_2,\ldots, S_{N_1}$, and $T$ into $N_2 := \lceil |T|/\beta\rceil$ subsets $T_1,\ldots,T_{N_2}$, so that the following conditions are satisfied:
\begin{multicols}{2}
\begin{itemize}
\item[(i)] $S_1 \leq \ldots \leq S_{N_1-1}$.

\item[(ii)] $|S_i|= \beta$ for $1\leq i\leq N_1-1$ and $|S_{N_1}|\leq \beta$.

\item[(iii)] $\left|\displaystyle{\iint}_{S_{N_1}\times T} (w-R_w^{\alpha})\, dxdy\right|\leq \dfrac{\| w-R_w^{\alpha}\|_{\Box}}{N_1}$. 

\item[(i)]$T_1 \leq \ldots \leq T_{N_2-1}.$

\item[(ii)] $|T_j|= \beta$ for $1\leq j\leq N_2-1$ and $|T_{N_2}|\leq \beta$.

\item[(iii)] $\left|\displaystyle{\iint}_{S\times T_{N_2}} (w-R_w^{\alpha})\, dxdy\right|\leq \dfrac{\| w-R_w^{\alpha}\|_{\Box}}{N_2}.$
\end{itemize}
\end{multicols}
To prove that a partition of $S$ satisfying the above three properties exists, we apply Lemma~\ref{lem:splitint} to the function $u(\cdot)=\int_{T} (w-R_w^{\alpha})(\cdot,y)dy$ and $P=S$; a similar proof shows that such a partition exists for $T$.

We will assume, without loss of generality, that $N_1,N_2\geq \teddydone{8}$, as we shall show the proposition holds true if $N_1\leq \teddydone{7}$ or $N_2\leq \teddydone{7}$. Assume that $N_1\leq \teddydone{7}$. Then, $|S|\leq \teddydone{7}\beta$, and so we get
\begin{equation}\label{eq:ngeq7}
    \|w-R_w^{\alpha}\|_{\Box}=\left|\iint(w-R_w^{\alpha})\mathbbm{1}_{S\times T}\right|\leq \|w-R_w^{\alpha}\|_\infty (\teddydone{7}\beta)\leq \|w\|_\infty \left(\teddydone{8}\alpha\right)
    = \teddydone{8}\|w\|_{\infty}^{\frac{2p-2}{3p-2}}\Lambda(w)^{\frac{2p}{5p-2}}.
\end{equation}
In the above inequalities, we used the fact that since $0\leq w,R_w^\alpha$ and $\|R_w^\alpha\|_\infty\leq \|w\|_\infty$, we have  $\|w-R_w^{\alpha}\|_\infty\leq\|w\|_\infty$.
To verify that the statement of the proposition holds true in this case, we need to show that the right hand side of \eqref{eq:ngeq7} is bounded above by the right hand side of \eqref{eq:approx-result}, which reduces to showing that
\begin{equation*}
\Lambda(w)^{\frac{p+2}{5p-2}}\leq \|w\|_{\infty}^{\frac{2}{3p-2}}.
\end{equation*}
However, since $\Lambda(w)\leq 1$ and $p>2$, we have $\Lambda(w)^{\frac{p+2}{5p-2}}\leq\Lambda(w)^{\frac{2}{3p-2}}$. This, together with the fact that $\Lambda(w)\leq \|w\|_\infty$, proves that Proposition \ref{prop:robinson-upprbnd} holds if $N_1 \leq \teddydone{7}$. An identical argument shows that the proposition also holds if $N_2 \leq \teddydone{7}$. From this point forward, we make the following assumption:
\begin{assumption}\label{assump1}
 $N_1,N_2 \geq \teddydone{8}.$
\end{assumption}
We now consider the value of $w-R_w^{\alpha}$ on the sets $S_i \times T_j$. Note that with repeated applications of the triangle inequality, we have
$$\sum_{\scriptsize{
\begin{array}{c}
1\leq i< N_1, \ 1\leq j< N_2\\
(S_i\times T_j)\cap \Delta\neq \emptyset\\
\end{array}}}\hspace{-2.12pt}
\Bigg|\iint_{S_i\times T_j} (w-R_w^{\alpha})\Bigg|\geq \Bigg|\iint_{(S\times T)\cap \Delta} (w-R_w^{\alpha})\Bigg|
-\Bigg|\iint_{(S_{N_1}\times T)\cup (S\times T_{N_2})} (w-R_w^{\alpha})\Bigg|,$$
which, together with Equation~\eqref{eq;choice-D}  and property (iii) of the partitions $\{S_i\}$ and $\{T_i\}$, implies that
\begin{equation*}
\sum_{\scriptsize{
\begin{array}{c}
1\leq i< N_1, \ 1\leq j< N_2\\
(S_i\times T_j)\cap \Delta\neq \emptyset\\
\end{array}}}\hspace{-2.12pt}
\Bigg|\iint_{S_i\times T_j} (w-R_w^{\alpha})\Bigg|\geq \Big(\frac{1}{2}-\frac{1}{N_1}-\frac{1}{N_2}\Big)\| w-R_w^{\alpha}\|_{\Box}.
\end{equation*}
The above inequality, combined with Assumption~\ref{assump1}, implies that
\begin{equation}\label{eq:int-over-sq-grid}
\sum_{\scriptsize{
\begin{array}{c}
1\leq i< N_1, \ 1\leq j< N_2\\
(S_i\times T_j)\cap \Delta\neq \emptyset\\
\end{array}}}\hspace{-2.12pt}
\Bigg|\iint_{S_i\times T_j} (w-R_w^{\alpha})\Bigg|\geq \teddydone{\frac{1}{4}}\| w-R_w^{\alpha}\|_{\Box}.
\end{equation}
\begin{claim}\label{claim:si0ti0ineq}
  There exists a cell $S'\times T'\subseteq \Delta$ contained in a region $\mathcal{R}_k=\B_k\cap \V_{k+1}$ so that $|S'|=|T'|=\alpha$, and
\begin{equation*}
\Big| \iint_{S'\times T'} (w-R_w^{\alpha})dxdy\Big| \geq  \alpha^2\left[\teddydone{\frac{1}{4}}\| w-R_w^{\alpha}\|_{\Box}
-2(2mM-1)\|w-R_w^{\alpha}\|_p\alpha^{1-\frac{2}{p}}\right].
\end{equation*}
\end{claim}
We use the pigeonhole principle to prove the claim. By Lemma~\ref{lem:rk-parts-delt/g}, $\Delta=\bigcup_{k=0}^{mM-1}\mathcal{R}_k\, \cup  \, \bigcup_{k=1}^{mM-1} \G_k$, and each of the regions $\G_k$ or $\mathcal{R}_k$ is bounded by \IN{boundary curve}s  in
$\{f_k, g_l: \ 1\leq k\leq mM-1,\, 1\leq l\leq mM\}$. Thus, if a  cell $S_i\times T_j$  does not cross the \IN{graph} of any of these \IN{boundary curve}s, then it must be entirely contained inside one closed region $\overline{\mathcal{R}_k}$ or $\overline{\G_k}$. 
Next, by Lemma~\ref{lem:grey-regions-small}, no grey regions $\overline{{\G_k}}$ can contain any ${S_i}\times {T_j}$ with $1\leq i<N_1$ and $1\leq j<N_2$, because $|S_i|=|T_j|=\beta$ for such $i$ and $j$. Thus, these cells must either lie in a single region $\overline{\mathcal{R}_k}$ or cross a \IN{boundary curve}. Let $\I$ denote the collection of indices $(i,j)$ with $i<N_1$ and $j<N_2$, for which the associated cells $S_i \times T_j$ do not lie in a single region $\overline{\mathcal{R}_k}$. From above,
$$\I=\Big\{(i,j): 1\leq i<N_1, 1\leq j< N_2,~\mbox{and}~\exists\  1\leq k \leq mM-1~\mbox{s.t.}~(S_i\times  T_j)~\mbox{crosses}~f_k~\mbox{or}~g_{k}~\mbox{or } g_{mM}\Big\}.$$
Because the lower and upper boundaries $f_k,g_k$ are increasing curves, each $f_k$ and $g_k$ cross at most $2/\beta$ cells from the grid. As there are $2mM-1$ total $f_k$ and $g_k$, we have 
$$|\I|\leq \frac{2(2mM-1)}{\beta}.$$
Using H\"{o}lder's inequality and that every cell indexed in $\I$ is of size $\beta^2$, we get 
\begin{equation}\label{eq:over-I}
\sum_{\scriptsize{
\begin{array}{c}
(i,j)\in \I \\
(S_i\times T_j)\cap \Delta\neq \emptyset\\
\end{array}}}
\left|\iint_{S_i\times T_j} (w-R_w^{\alpha})\, dxdy\right|\leq \frac{2(2mM-1)}{\beta} \|w-R_w^{\alpha}\|_p\beta^{2-\frac{2}{p}}. 
\end{equation}
Putting inequalities \eqref{eq:int-over-sq-grid} and \eqref{eq:over-I} together, it must be true that
\begin{equation*}
\sum_{\scriptsize{
\begin{array}{c}
1\leq i< N_1, \ 1\leq j< N_2\\
(S_i\times T_j)\cap \Delta\neq \emptyset\\
(i,j)\not\in {\cal I}
\end{array}}}
\left|\iint_{S_i\times T_j} (w-R_w^{\alpha})\, dxdy\right|
\geq \teddydone{\frac{1}{4}}\| w-R_w^{\alpha}\|_{\Box}
- 2(2mM-1) \|w-R_w^{\alpha}\|_p\beta^{1-\frac{2}{p}}. 
\end{equation*}
 Since there are at most $(\lfloor \beta^{-1}\rfloor)^2$  cells  $S_i\times T_j$ of size $\beta\times \beta$ and $\beta^2 \leq (\lfloor \beta^{-1}\rfloor)^{-2}$, there must exist a cell $S_{i_0}\times T_{j_0}\subseteq \Delta$ so that $|S_{i_0}|=|T_{j_0}|=\beta$, $(i_0,j_0)\not\in \I$, and
\begin{equation*}
\left| \iint_{S_{i_0}\times T_{j_0}} (w-R_w^{\alpha})dxdy\right| \geq  \beta^2\left[\teddydone{\frac{1}{4}}\| w-R_w^{\alpha}\|_{\Box}
- 2(2mM-1) \|w-R_w^{\alpha}\|_p\beta^{1-\frac{2}{p}}\right].
\end{equation*}
So $S_{i_0}\times T_{j_0}$ lies entirely in $\overline{\mathcal{R}_k}=\overline{\B_k\cap \V_{k+1}}$ for some $0\leq k\leq mM-1$. By Lemma \ref{lem:gnrlpigeon}, we can reduce $S_{i_0}\times T_{j_0}$ to an $\alpha \times \alpha$ cell, called $S'\times T'$, contained in $\mathcal{R}_k=\B_k\cap \V_{k+1}$ satisfying
\begin{equation*}
\left| \iint_{S'\times T'} (w-R_w^{\alpha})\, dxdy\right| \geq  \alpha^2\left[\teddydone{\frac{1}{4}}\| w-R_w^{\alpha}\|_{\Box}-
{2(2mM-1)}\|w-R_w^{\alpha}\|_p\beta^{1-\frac{2}{p}}\right].
\end{equation*}
This inequality holds for all $\beta > \alpha$; taking $\beta \to \alpha$ proves Claim \ref{claim:si0ti0ineq}.

The cell $S' \times T'$ is thus contained in some region $\mathcal{R}_k$. From the definition of $R_w^{\alpha}$ and the black and white regions, we observe that if $1\leq k\leq Mm-1$, then $\frac{k}{m}< R_w^{\alpha}\leq\frac{k+1}{m}$ on $S'\times T'$, and if $k=0$, then $0\leq R_w^{\alpha}\leq\frac{1}{m}$ on $S'\times T'$.

\begin{claim}\label{claim-3case}
Under the assumptions made so far, 
\begin{equation*}
\| w-R_w^{\alpha}\|_{\Box}\leq \teddydone{4}\left(\frac{2\Lambda(w)}{\alpha^2}+ \frac{1}{m}+{2(2mM-1)} \|w-R_w^{\alpha}\|_p\alpha^{1-\frac{2}{p}}\right).
\end{equation*}
\end{claim}
\noindent To prove the claim, we consider three cases:
\begin{description}
\item[{\bf Case 1:}] Assume that $\iint_{S'\times T'} (w-R_w^{\alpha})\, dxdy>0$ and $0\leq k\leq mM-2$. 

Using Claim~\ref{claim:si0ti0ineq} and the fact that $|S'\times T'|=\alpha^2$, we have
\begin{equation}\label{eqn:avg}
\overline{w}(S'\times T')-\frac{k}{m}\geq \overline{w-R_w^{\alpha}}(S'\times T')
\geq
\teddydone{\frac{1}{4}}\| w-R_w^{\alpha}\|_{\Box}- 2(2mM-1) \|w-R_w^{\alpha}\|_p\alpha^{1-\frac{2}{p}}.
\end{equation}
Now let $(x,y)$ be the lower right corner of $S'\times T'$.  Then $(x,y)\in \V_{k+1}$, implying that $\LR (x,y)$ contains a region $S_a\times T_b$ so that $|S_a|=|T_b|=\alpha$ and that $\overline{w}(S_a \times T_b)\leq  \frac{k+1}{m}$. So  inequality \eqref{eqn:avg} combined with the definition of $\Lambda$ implies that 
\begin{eqnarray*}
    \Lambda(w)&\geq& \frac{\alpha^2}{2}\bigg(\overline{w}(S'\times T')-\overline{w}(S'\times T_b)+\overline{w}(S'\times T_b)-\overline{w}(S_a\times T_b)\bigg)\\
&\geq& \frac{\alpha^2}{2}\bigg(\teddydone{\frac{1}{4}}\| w-R_w^{\alpha}\|_{\Box}- 2(2mM-1) \|w-R_w^{\alpha}\|_p\alpha^{1-\frac{2}{p}}+ \frac{k}{m}-\frac{k+1}{m}\bigg),
\end{eqnarray*}
which reduces to 
\begin{eqnarray*}
\teddydone{\frac{1}{4}}\| w-R_w^{\alpha}\|_{\Box} 
\leq \frac{2\Lambda(w)}{\alpha^2}+\frac{1}{m}+2(2mM-1)\|w-R_w^{\alpha}\|_p\alpha^{1-\frac{2}{p}}.
 \end{eqnarray*}

\item[{\bf Case 2:}] Assume $\iint_{S'\times T'}(w-R_w^{\alpha})\, dxdy\leq 0$ and $1\leq k\leq mM-1$. 

By a similar argument used to show \eqref{eqn:avg},
\begin{equation*}
\frac{k+1}{m}-\overline{w}(S'\times T')\geq \teddydone{\frac{1}{4}}\| w-R_w^{\alpha}\|_{\Box}- 2(2mM-1) \|w-R_w^{\alpha}\|_p\alpha^{1-\frac{2}{p}}.
\end{equation*}
Now let $(x,y)$ be the upper left corner of $S'\times T'$. Then $(x,y)\in \B_{k}$, which means $\UL (x,y)$ contains a region $S_c\times T_d$ such that $|S_c|=|T_d|=\alpha$ and $\overline{w}(S_c \times T_d)> \frac{k}{m}$. Using the definition of \IN{$\Lambda$}, similar to the argument in Case 1, we get
\begin{eqnarray*}
\Lambda(w)&\geq& \frac{\alpha^2}{2}\left(\frac{k}{m}+ \teddydone{\frac{1}{4}}\| w-R_w^{\alpha}\|_{\Box}- 2(2mM-1) \|w-R_w^{\alpha}\|_p\alpha^{1-\frac{2}{p}} - \frac{k+1}{m}\right),
\end{eqnarray*}
which reduces to 
\begin{eqnarray*}
 \teddydone{\frac{1}{4}}\| w-R_w^{\alpha}\|_{\Box}\leq \frac{2\Lambda(w)}{\alpha^2}+ \frac{1}{m}+{2(2mM-1)} \|w-R_w^{\alpha}\|_p\alpha^{1-\frac{2}{p}}.
\end{eqnarray*}

\item[\textbf{Case 3:}] Assume that either $\iint_{S'\times T'} (w-R_w^{\alpha})\, dxdy>0$ and $k= mM-1$ \textit{or} that $\iint_{S'\times T'} (w-R_w^{\alpha})\, dxdy\leq 0$ and $k=0$. 

In the first assumption, we have that $R_w^{\alpha} > M-\frac{1}{m}$ on $S'\times T'$ whereas $w \leq M$ by definition. Thus, $\frac{\alpha^2}{m}\geq\iint_{S'\times T'} (w-R_w^{\alpha})\, dxdy>0$. In the second assumption, we have that $R_w^{\alpha} \leq \frac{1}{m}$ on $S'\times T'$ and by negativity of the integral in the second assumption it must be that $\frac{\alpha^2}{m} \geq \iint_{S'\times T'}(R_w^{\alpha}-w)\geq0$. Combining either of these results with Claim~\ref{claim:si0ti0ineq} gives 
$$
\frac{\alpha^2}{m} \geq \alpha^2\left[\teddydone{\frac{1}{4}}\| w-R_w^{\alpha}\|_{\Box}- 2(2mM-1)\|w-R_w^{\alpha}\|_p\alpha^{1-\frac{2}{p}}\right],
$$
which can be rearranged to show
$$
\teddydone{\frac{1}{4}}\| w-R_w^{\alpha}\|_{\Box}\leq \frac{1}{m}+{2(2mM-1)} \|w-R_w^{\alpha}\|_p\alpha^{1-\frac{2}{p}}.
$$
\end{description}
So Claim \ref{claim-3case} holds in all cases and we can now finish the proof. Taking $\alpha=\|w\|_{\infty}^{-\frac{p}{3p-2}}\Lambda(w)^{\frac{2p}{5p-2}}$ and $m=\lceil\Lambda(w)^{-\frac{p-2}{5p-2}}\rceil$ in Claim~\ref{claim-3case}, we get 
\begin{equation*}
\teddydone{\frac{1}{4}}\| w-R_w^{\alpha}\|_{\Box} \leq  \bigg(1+(4\|w-R_w^{\alpha}\|_p+2)\|w\|_{\infty}^{\frac{2p}{3p-2}}\bigg)\Lambda(w)^{\frac{p-2}{5p-2}}+\bigg(4\|w-R_w^{\alpha}\|_p\|w\|_{\infty}^{\frac{2p}{3p-2}}\bigg)\Lambda(w)^{\frac{2p-4}{5p-2}}.
 \end{equation*}
Since $\Lambda(w)\leq 1$, we  get $\Lambda(w)^{\frac{2p-4}{5p-2}}\leq \Lambda(w)^{\frac{p-2}{5p-2}}$; this simplifies the above equation to the desired result.
\end{proof}
As a corollary to Proposition \ref{prop:robinson-upprbnd}, we obtain an improvement for the earlier results of \cite{ghandehari2020graph} on Robinson approximation of graphons and kernels. 
\begin{corollary}\label{cor:bnddstability}
Let $w:[0,1]^2\to [0,1]$ be a graphon and $u:[0,1]^2\to\R$ be a kernel. Then 
$$\| w-R_w^{\alpha}\|_{\Box} \leq \teddydone{44}\Lambda(w)^{\frac{1}{5}} \qquad \text{and} \qquad \| u-R_u^{\alpha^*}\|_{\Box} \leq  \teddydone{44}\|u\|_\infty^{\frac{4}{5}}\Lambda(u)^{\frac{1}{5}},$$
where $\alpha= \|w\|_{\infty}^{-\frac{1}{3}}\Lambda(w)^{\frac{2}{5}}$ and $\alpha^* = \|u\|_{\infty}^{-\frac{2}{5}}\Lambda(u)^{\frac{2}{5}}.$
\end{corollary}
\begin{proof}
For $w \in \W_0$, let $\alpha_p = \|w\|_{\infty}^{-\frac{p}{3p-2}}\Lambda(w)^{\frac{2p}{5p-2}}$. Then, by Proposition \ref{prop:robinson-upprbnd}, we have  $\| R_w^{\alpha_p}-w\|_{\Box} \leq  \teddydone{44}\Lambda(w)^{\frac{p-2}{5p-2}},$ since $\|w\|_\infty\leq 1$ and $\|w-R_w^{\alpha_p}\|_p\leq 1$. However, as $w \in L^p[0,1]^2$ for every $p\geq1$, we can allow $p\rightarrow \infty$, showing
$\| w-R_w^{\alpha}\|_{\Box} \leq  \teddydone{44}\Lambda(w)^{\frac{1}{5}},$ where $\alpha = \|w\|_{\infty}^{-\frac{1}{3}}\Lambda(w)^{\frac{2}{5}}$ as desired. For $u\in \W$, we scale by $\|u\|_{\infty}$ to make a new function $u^* := u/\|u\|_{\infty}$ to get that %
\begin{equation}\label{eq:scaledstable}
    \|u^*-R_{u^*}^{\alpha^*}\|_{\Box} \leq \teddydone{44}\Lambda(u^*)^{\frac{1}{5}},
\end{equation}
where $\alpha^* = \|u\|_{\infty}^{-\frac{2}{5}}\Lambda(u)^{\frac{2}{5}}$. We note that by definition of \IN{$\Lambda$} and positivity of $\|u\|_{\infty}$, it must be the case that $\Lambda(u^*)^{\frac{1}{5}}= \|u\|_{\infty}^{-\frac{1}{5}}\Lambda(u)^{\frac{1}{5}}.$ Furthermore, by definition of $R_u^{\alpha},$ we have $R_{u^*}^{\alpha^*} = \|u\|_{\infty}^{-1}R_u^{\alpha^*}.$ Thus, combining these two observations with \eqref{eq:scaledstable} yields
$\frac{1}{\|u\|_{\infty}}\|u-R_u^{\alpha^*}\|_{\Box}\leq \teddydone{44}\|u\|_{\infty}^{-\frac{1}{5}}\Lambda(u)^{\frac{1}{5}},$
which proves the claim.
\end{proof}
We are now ready to state and prove our recovery result about \IN{$\Lambda$}. We begin with a necessary definition---a key technique in this proof is taking an unbounded graphon and ``cutting it off'' at a certain threshold.
\begin{notation}[$M$-cut-off]\label{def:mcutoff}
Let $w \in \mathcal{W}^1$ and define
\begin{equation*}
E_M=\{(x,y)\in [0,1]^2: w(x,y)>M\}.
\end{equation*}
The $M$-\emph{cut-off} of $w$, denoted by $w_M$, is defined to be $w_M:=(\mathbbm{1}-\mathbbm{1}_M)w$, where $\mathbbm{1}_M$ is the characteristic function of $E_M$ and $\mathbbm{1}$ is the characteristic function of $[0,1]^2$. 
\end{notation}
\begin{theorem}\label{thm:main-result}
Suppose $w:[0,1]^2\to [0,\infty)$ is an $L^p$-\IN{kernel} with $p > 5$ and $\|w\|_p\leq1$. Then there exists some $\alpha \in [0,\frac{1}{2})$ such that $R_w^{\alpha}$, the \IN{Robinson approximation} of $w$ with parameter $\alpha$, satisfies 
\begin{equation}\label{eq:main_result}
    \|w-R_w^{\alpha}\|_{\Box} \leq 78\Lambda(w)^{\frac{p-5}{5p-5}}.
\end{equation}
\end{theorem}
The idea of the proof is as follows: \mahya{we use triangle inequality to bound $\|w-R_w^{\alpha}\|_{\Box}$ by the sum of the cut-norms of $w-w_M$, $w_M-R_{w_M}^{\alpha}$ and $R_{w_M}^{\alpha}-R_{w}^\alpha$, and then bound each term using $\Lambda$.} If $w_M$ is Robinson, then the upper bound is proved without use of Corollary \ref{cor:bnddstability}. If $w_M$ is not Robinson, then $\|w_M-R_{w_M}^{\alpha}\|_{\Box}$ must be handled using Corollary \ref{cor:bnddstability}. Handling these two cases finishes the proof. 
\begin{proof}
By definition of Robinson approximation, if $\Lambda(w) = 0$, then we set $\alpha=0$, resulting in $R_w^{\alpha}=w$. Thus, we assume that $\Lambda(w) > 0$ and define our cut-off value $M = 2\Lambda(w)^{-\frac{1}{p-1}}$. We now consider two cases.

\noindent\textbf{Case 1:} Suppose that $\Lambda(w_M) >0$. 

Let $\alpha=\|w_M\|_{\infty}^{-\frac{2}{5}}\Lambda(w_M)^{\frac{2}{5}}$. The triangle inequality can then be used to show that 
\begin{equation*}
\|w-R_w^{\alpha}\|_{\Box} \leq \|w-w_M\|_{\Box}+\|w_M-R_{w_M}^{\alpha}\|_{\Box}+\|R_w^{\alpha}-R_{w_M}^{\alpha}\|_{\Box}.
\end{equation*}
We will proceed by bounding each of the terms on the right hand side one by one, starting with $\|w-w_M\|_{\Box}$. 
By definition, $\mathbbm{1}_M$ is the characteristic function of $E_M$, the region of $[0,1]^2$ where $w > M$. Since $M|E_M|^{\frac{1}{p}}\leq \|w\mathbbm{1}_M\|_p\leq \|w\|_p\leq 1$, we get $|E_M|\leq (\frac{1}{M})^p$. Therefore, for $q$ satisfying  $1/p+1/q=1$, it is true that
\begin{equation}\label{eq:qnormM}
\|\mathbbm{1}_M\|_q=|E_M|^{1/q}\leq\left(\frac{1}{M}\right)^{\frac{p}{q}}=M^{1-p}= 2^{1-p}\Lambda(w).
\end{equation}
It is also true that 
\begin{equation}\label{eq:w-wm-ineq}
    \|w-w_M\|_{\Box}\leq\|w-w_M\|_1 =\|w\mathbbm{1}_M\|_1\leq \|w\|_p\|\mathbbm{1}_M\|_q \leq 2^{1-p}\Lambda(w)
    \leq 2^{1-p}\Lambda(w)^{\frac{p-5}{5p-5}},
\end{equation}
handling the first term. We can further say that $\|w-w_M\|_{\Box}\leq\frac{1}{16}\Lambda(w)^{\frac{p-5}{5p-5}}$ for all $p > 5$. Now we shift focus to $\|w_M-R_{w_M}^{\alpha}\|_{\Box}$. By Corollary \ref{cor:bnddstability}, as $w_M$ is bounded, 
$$\|w_M-R_{w_M}^{\alpha}\|_{\Box} \leq  \teddydone{44}\|w_M\|_{\infty}^{\frac{4}{5}}\Lambda(w_M)^{\frac{1}{5}}
    \leq \teddydone{44}M^{\frac{4}{5}}(\Lambda(w)+2^{1-p}\Lambda(w))^{\frac{1}{5}},$$
where the second inequality is due to the combination of \eqref{eq:Lambda-prop}
for  $w_M=w-w\mathbbm{1}_M$ alongside the fact that  $\Lambda(-w\mathbbm{1}_M)\leq \|w\mathbbm{1}_M\|_1\leq 2^{1-p}\Lambda(w)$ by \eqref{eq:qnormM}. 
This simplifies to 
\begin{align}\label{eq:rwm-wm-ineq}
    \|w_M-R_{w_M}^{\alpha}\|_{\Box}\leq \teddydone{44} (1+2^{1-p})^{\frac{1}{5}}M^{\frac{4}{5}}\Lambda(w)^{\frac{1}{5}} 
    \leq \teddydone{44} (1+2^{1-p})^{\frac{1}{5}}\ 2^{\frac{4}{5}}\Lambda(w)^{\frac{p-5}{5p-5}}\leq \teddydone{77.6}\Lambda(w)^{\frac{p-5}{5p-5}},
\end{align}
where we used the fact that $p > 5$.

To bound the third term, we use Proposition~\ref{prop:rw-properties} (i) and Proposition~\ref{prop:Lambda-prop} (i) to get that
\begin{equation*}
\|R_w^{\alpha}-R_{w_M}^{\alpha}\|_{\Box}\leq \|R_w^{\alpha}-R_{w_M}^{\alpha}\|_{\infty}\leq \|R_{w-w_M}^{\alpha}\|_\infty\leq \alpha^{-2}\|w-w_M\|_1=\alpha^{-2}\|w\mathbbm{1}_M\|_1\leq \alpha^{-2}\|w\|_p\|\mathbbm{1}_M\|_q,
\end{equation*}
where, in the last inequality, we used H\"{o}lder's inequality with conjugate indices $p,q$. Now, using the upper bound for $\|\mathbbm{1}_M\|_q$ provided by \eqref{eq:qnormM}, and substituting the value of $\alpha$, it is true that
\begin{equation}\label{eq:mid-way}
    \|R_w^{\alpha}-R_{w_M}^{\alpha}\|_{\Box}\leq 2^{1-p}\alpha^{-2}\Lambda(w)=2^{1-p}\Lambda(w)\left(\|w_M\|_{\infty}^{-\frac{2}{5}}\Lambda(w_M)^{\frac{2}{5}}\right)^{-2}.
\end{equation}
Applying \eqref{eq:Lambda-prop} to $w=w_M+w\mathbbm{1}_M$ gets us
\begin{equation*}
\Lambda(w_M) \geq \Lambda(w)-\Lambda(w\mathbbm{1}_M) \geq \Lambda(w)-\|w\mathbbm{1}_M\|_1 \geq \Lambda(w)-\|w\|_p\|\mathbbm{1}_M\|_q \geq (1-2^{1-p})\Lambda(w),
\end{equation*}
and since $p>5$, this requires that $\Lambda(w_M) \geq \frac{15}{16}\Lambda(w)$. Using this together with $\|w_M\|_{\infty}\leq M=2\Lambda(w)^{\frac{-1}{p-1}}$ and inequality \eqref{eq:mid-way} implies that 
\begin{align*}
 2^{p-1}\|R_w^{\alpha}-R_{w_M}^{\alpha}\|_{\Box}&\leq \Lambda(w)\|w_M\|_{\infty}^{\frac{4}{5}}\Lambda(w_M)^{-\frac{4}{5}}
\leq \Lambda(w)M^{\frac{4}{5}}(1-2^{1-p})^{-\frac{4}{5}}\Lambda(w)^{-\frac{4}{5}}
\leq \frac{2^{\frac{4}{5}}}{(1-2^{1-p})^{\frac{4}{5}}}\Lambda(w)^{\frac{p-5}{5p-5}}.
\end{align*}
When $p > 5$, we have $2^{\frac{4}{5}}2^{1-p}(1-2^{1-p})^{-\frac{4}{5}} \leq 0.2$, so 
\begin{equation}\label{eq:rw-rw-ineq}
    \|R_w^{\alpha}-R_{w_M}^{\alpha}\|_{\Box} \leq 0.2\Lambda(w)^{\frac{p-5}{5p-5}}.
\end{equation}
Thus \eqref{eq:w-wm-ineq}, \eqref{eq:rwm-wm-ineq}, and \eqref{eq:rw-rw-ineq} together imply that $\|w-R_w^{\alpha}\|_{\Box} \leq \teddydone{78}\Lambda(w)^{\frac{p-5}{5p-5}},$ proving the statement of the theorem in this case.

\noindent\textbf{Case 2:} Suppose that $\Lambda(w_M) = 0$. 

Let $\alpha = M^{-\frac{2}{5}}\Lambda(w)^{\frac{2}{5}}.$ We proceed similarly to Case 1, bounding each of the three summands in the right hand side of the following inequality:
\begin{equation*}
\|w-R_w^{\alpha}\|_{\Box} \leq \|w-w_M\|_{\Box}+\|w_M-R_{w_M}^{\alpha}\|_{\Box}+\|R_w^{\alpha}-R_{w_M}^{\alpha}\|_{\Box}.
\end{equation*}
The first term on the right side of the inequality can be handled identically to Case 1, yielding $\|w-w_M\|_{\Box}\leq\frac{1}{16}\Lambda(w)^{\frac{p-5}{5p-5}}$. For the third term, we can proceed identically to Case 1 and get
$$\|R_w^{\alpha}-R_{w_M}^{\alpha}\|_{\Box}\leq 2^{1-p}\alpha^{-2}\Lambda(w).$$
Substituting $\alpha=M^{-\frac{2}{5}}\Lambda(w)^{\frac{2}{5}}$ and noting that $p>5$ results in $\|R_w^{\alpha}-R_{w_M}^{\alpha}\|_{\Box}\leq 0.2 \Lambda(w)^{\frac{p-5}{5p-5}}$.

However, the second term $\|w_M-R_{w_M}^{\alpha}\|_{\Box}$ must be handled differently. Because $\Lambda(w_M) = 0$ implies that $w_M$ is Robinson a.e., we will directly approximate $R_{w_M}^{\alpha}$. 

\begin{claim}\label{claim-case2-approx}
When $w_M:[0,1]^2\to [0,M]$ is Robinson a.e., we have $$\|w_M-R_{w_M}^{\alpha}\|_\Box \leq \iint_{\{|x-y|\leq2\alpha\}}w_M(x,y)\, dxdy.$$
\end{claim}
\noindent To prove the claim, first note that for any point $(x,y) \in \Delta$, we have that $w_M(x,y)$ is an upper bound for every value of $w_M$ over the set $\UL(x,y)$; thus, any average over that set (such as in the definition of $R_w^{\alpha}$) would not exceed $w_M(x,y)$. So $R_{w_M}^{\alpha} \leq w_M$. On the other hand, as $w_M$ is Robinson a.e., we have 
\begin{equation}\label{eq:rwm}
R_{w_M}^{\alpha}(x,y) = R_{w_M}^{\alpha}(y,x)=\begin{cases} \dfrac{1}{\alpha^2}\displaystyle\iint_{[x-\alpha,x]\times[y,y+\alpha]}w_Mdxdy & \text{ for } (x,y) \in [\alpha,1-\alpha]^2\cap \Delta\\  0 & \text{ otherwise }\end{cases}.
\end{equation}
We now introduce an auxiliary function $\widetilde{w}_M$ defined as follows:
\begin{equation*}
    \widetilde{w}_M(x,y) = \widetilde{w}_M(y,x) = \begin{cases} w_M(x-\alpha,y+\alpha) & (x,y) \in [\alpha,1-\alpha]^2 \cap \Delta \\ 0 & \text{ otherwise }\end{cases}.
\end{equation*}
Since $w_M$ is Robinson, we have that $R_{w_M}^{\alpha} \geq \widetilde{w}_M$. Thus we have $0\leq w_M - R_{w_M}^{\alpha} \leq w_M - \widetilde{w}_M$ pointwise, allowing us to show the following:
\begin{align*}
    \|w_M - R_{w_M}^{\alpha}\|_{\Box} &\leq  \|w_M - \widetilde{w}_M\|_{\Box} = 2\iint_{[0,1]^2\cap \Delta}(w_M-\widetilde{w}_M)dxdy \\
    &= 2\left(\iint_{[0,1]^2\cap \Delta}w_Mdxdy - \iint_{[\alpha,1-\alpha]^2\cap\Delta}\widetilde{w}_Mdxdy\right) \\
    &= 2\left(\iint_{[0,1]^2\cap \Delta}w_Mdxdy - \iint_{0\leq x\leq y-2\alpha\leq 1-2\alpha}w_M (x,y)dxdy\right)\\
    &\leq \iint_{\{|x-y|\leq2\alpha\}}w_M(x,y)dxdy.
\end{align*}
This proves Claim~\ref{claim-case2-approx}. 
Finally, note that
\begin{equation*}
\iint_{\{|x-y|\leq 2\alpha\}}w_M(x,y)dxdy \leq \|w_M\|_p\|\mathbbm{1}_{[0,1]^2\cap \{|x-y|\leq2\alpha\}}\|_q \leq 
(1-(1-2\alpha)^2)^{\frac{1}{q}} \leq (4\alpha)^{1-\frac{1}{p}},
\end{equation*}
where in the last inequality we used that $0\leq \alpha\leq 1$. Substituting for $\alpha$ and noting that $p>5$, we have 
$$\|w_M - R_{w_M}^{\alpha}\|_{\Box} \leq 4^{1-\frac{1}{p}}\Big(\frac{\Lambda(w)}{M}\Big)^{(\frac{2}{5})(1-\frac{1}{p})}
\leq 4{\Lambda(w)}^{\frac{2}{5}}.$$
Since $\Lambda(w)\leq 1$ and  $\frac{2}{5} \geq (p-5)(5p-5)^{-1}$ for all $p>5$, we get 
$\|w_M - R_{w_M}^{\alpha}\|_{\Box} \leq  4{\Lambda(w)}^{\frac{2}{5}}\leq 4\Lambda(w)^{\frac{p-5}{5p-5}}.$
Therefore, if $\Lambda(w_M)=0$, we have that $\|w-R_w^{\alpha}\|_{\Box} \leq 5\Lambda(w)^{\frac{p-5}{5p-5}}$, proving the statement of the theorem in this case.
\end{proof}
\mahya{
\begin{remark}[Concluding remarks]  
We finish this section with a discussion on the performance of $\Lambda$ and a comparison with the previously defined function $\Gamma$ from \cite{Chuangpishit_2015}.
\begin{itemize}
\item[(i)] The function $\Gamma$ provides a continuous and robust mapping on $L^p$-graphons only for the case $p=\infty$ (see Remark~\ref{remark:teddy}), whereas $\Lambda$ can be used as a suitable gauge of Robinson property for any $L^p$-graphon with $5<p\leq \infty$ (see Theorem \ref{thm:main-result}). Moreover, the error (in cut-norm) of the Robinson approximation provided by $\Gamma$ of a graphon $w:[0,1]^2\to [0,1]$ is bounded by $14\Gamma(w)^{\frac{1}{7}}$ (see \cite[Theorem 3.2]{Chuangpishit_2015}). However, using $\Lambda$ for Robinson approximation of $w$ leads to a much improved upper bound of $44\Lambda(w)^{\frac{1}{5}}$ (see Corollary~\ref{cor:bnddstability}). 
This improvement is partly due to the refined definition of $\Lambda$ (compared to $\Gamma$), 
and partly due to the possibility of using sharper results in Proposition~\ref{prop:robinson-upprbnd} for $L^p$-graphons to obtain bounds for $L^\infty$-graphons by letting $p\to \infty$.
\item[(ii)] Theorem \ref{thm:main-result} holds only for $p>5$. This is an artefact of the proof technique used and not an indication that similar results do not hold for $1 \leq p \leq 5$. We anticipate that ideological bottlenecks would appear at $p\leq 2$ and $p=1$. We conjecture that results analogous to Theorem \ref{thm:main-result} are valid for $p\geq 2$, though achieving these results would likely require new proof methods and more advanced techniques.
Indeed, our proof of Theorem~\ref{thm:main-result} is based on the idea of approximating an $L^p$-graphon with its $M$-cut-off, and then applying Robinson approximation results for bounded graphons to the cut-off function. The condition $p>5$ is essential for obtaining reasonable cut-off approximations (see \eqref{eq:w-wm-ineq}). 
This requirement is the primary impediment to adapting the proof of Theorem~\ref{thm:main-result} to $p\leq 5$.
\end{itemize}
\end{remark}
}

\section*{Acknowledgements}
We thank Jeannette Janssen for helpful comments which improved this work. We thank the Department of Mathematical Sciences at the University of Delaware for their continued support throughout the process of this research. 
M.~Ghandehari was supported by NSF grant DMS-1902301 while this work was being completed.
\mahya{Finally, we sincerely thank the anonymous reviewer for reading the
manuscript and suggesting improvements.}

\begin{appendices}
\section{Proofs of Lemmas~\ref{lem:grey-regions-small}, \ref{lem:splitint} and \ref{lem:gnrlpigeon}}\label{section:proofs of lemmas}
In this Appendix we present proofs of the lemmas used previously in the paper, recalling their statements for clarity.
\mahya{
\begin{lemma*}[Lemma \ref{lem:grey-regions-small}]
Let $k\in{\mathbb Z}^{\geq 0}$, $w \in \W^p$ with $w\geq 0$, and $\alpha \in (0,1)$. Then, $\overline{\G_k}$ does not contain any $\beta\times \beta$ square, where $\beta>\alpha$. Here, $\overline{\G_k}$ denotes the closure of $\G_k$ in the Euclidean topology of ${\mathbb R}^2$. 
\end{lemma*}
\begin{proof}[Proof of Lemma \ref{lem:grey-regions-small}]
Let $k \in \Z^{\geq 0}$ be fixed and let $\mathcal{G}_k$ be defined with parameter $m$ and $\alpha$.
Towards a contradiction, let $\beta > \alpha$ and suppose there exist measurable subsets $S, T \subseteq [0, 1]$ with $|S| = |T | = \beta$ for which $S \times T \subseteq \overline{\mathcal{G}_k}$. 
Since $|S| = |T | = \beta>\alpha$, there exist $a_1,a_2\in {S}$ and $b_1,b_2\in {T}$ such that 
 $a_2 - a_1 >\alpha\ \mbox{ and } \ b_2 - b_1 >\alpha.$
Note that $(a_i, b_j) \in \overline{\mathcal{G}_k}$ for $i, j = 1, 2$.
Since every point in $\overline{\mathcal{G}_k}$ that is not on the lower or upper boundary curves must be an interior point of $\mathcal{G}_k$,
 by moving the points slightly if necessary, we can assume that 
$$a_2 - a_1 >\alpha,\ \ b_2 - b_1 >\alpha, \ \mbox{ and } (a_i, b_j) \in{\mathcal{G}_k}, i, j = 1, 2.$$

We shall now show the inclusion $(a_1, a_2) \times (b_1, b_2) \subseteq {\mathcal{G}_k}$. 
Towards a contradiction, suppose that there exists some point $(z,w) \in \UL(a_2,b_1)\cap\LR(a_1,b_2)\setminus \mathcal{G}_k$, which implies that either $(z,w) \in \UL(a_2,b_1)\cap\LR(a_1,b_2)\cap\mathcal{B}_k$ or that $(z,w) \in \UL(a_2,b_1)\cap\LR(a_1,b_2)\cap\mathcal{V}_k$. The first case implies that $(a_2,b_1) \in \mathcal{B}_k$ while the second case implies that $(a_1,b_2) \in \mathcal{V}_k$, both of which are contradictions with $(a_i, b_j) \in {\mathcal{G}_k}$. Thus, $\mathcal{G}_k$ contains $(a_1, a_2) \times (b_1, b_2)$.

 Clearly, $(a_1, a_2) \times (b_1, b_2)$ contains a closed $\alpha \times\alpha$ rectangle which we denote $[a'_1, a'_2] \times [b'_1, b'_2]$. The two points $(a'_1, b'_2)$ and $(a'_2, b'_1)$ are elements of $\mathcal{G}_k$, so they fail to satisfy the conditions for both $\mathcal{V}_k$ and $\mathcal{B}_k$. In particular, we have $\overline{w}([a'_1, a'_2]\times [b'_1, b'_2]) \leq (k-1)/m$ as $(a'_2, b'_1) \not \in \mathcal{B}_k$, as well as $\overline{w}([a'_1, a'_2], [b'_1, b'_2]) > (k-1)/m$ as $(a'_1, b'_2) \not \in \mathcal{V}_k$. These statements form a contradiction, showing the initial claim must be true.
\end{proof}}
\begin{lemma*} (Lemma~\ref{lem:splitint})
Let $u\in L^\infty([0,1])$, let $P\subseteq [0,1]$ be a measurable subset such that $\int_P u~dx\neq 0$, and let $0 < \beta <|P|$ be fixed. 
Then $P$ can be partitioned into $N:=\lceil |P|/\beta\rceil$ subsets $P_1,\ldots,P_{N}$ so that the following conditions are satisfied:
\begin{onehalfspacing}
\begin{itemize}
    \item[(i)] $P_1 \leq \ldots \leq P_{N-1}.$
    
    \item[(ii)] $|P_i|= \beta$ for $1\leq i\leq N-1$ and $|P_{N}|\leq \beta$.
    
    \item[(iii)] $\big|\int_{P_{N}} u~dx\big|\leq \frac{1}{N}\big|\int_{P}u~dx\big|$. 
\end{itemize}
\end{onehalfspacing}
\end{lemma*}
\begin{proof}[Proof of Lemma \ref{lem:splitint}]
For any sets $P_1,\ldots,P_N$ satisfying the above properties, we must have 
$$|P_N| =\delta := |P|-\beta\left(\left\lceil\frac{|P|}{\beta}\right\rceil-1\right).$$
To prove the lemma, it is enough to find a subset   $P_N\subseteq P$ satisfying $|P_N|=\delta$ and condition (iii). Indeed, given such a subset $P_N$, one can form the other sets $P_1,\ldots,P_{N-1}$ by splitting $P\setminus P_N$ into consecutive sets of measure $\beta$. If $\delta = \beta$, then the lemma follows trivially from the pigeonhole principle. So, we assume that $\delta < \beta$.
Moreover,  replacing $u$ by $-u$ if necessary, we can assume that $\int_P u~dx > 0$. 
We claim there must exist a set $Q \subset P$ such that $|Q| = |P|-\beta\left(\left\lceil \dfrac{|P|}{\delta}\right\rceil-1\right)$ and $\int_Q u~dx > 0$. To show this statement, note that
\begin{equation}\label{eq:Qpositivity}
    \int_P u~dx = \int_{P^+}u~dx + \int_{P^-}u~dx >0,
\end{equation}
where $P^+ := \{x \in P : u(x) > 0\}$ and $P^- :=\{y \in P : u(y) < 0\}$. If $|P^+|\geq \delta$, then any subset $Q$ of measure $\delta$ from $P^+$ suffices. If $|P^+| < \delta$, then we let $Q=P^+\cup Q'$ for a subset  $Q'\subseteq P^-$ with $|Q'|=\delta-|P^+|$, and note that 
 $\int_{Q} u~dx\geq \int_{P} u~dx>0$ by \eqref{eq:Qpositivity}. We now prove the existence of a subset $P_N\subseteq P$ satisfying $|P_N|=\delta$ and condition (iii). Towards a contradiction, assume that for any set $S \subseteq P$ such that $|S| = \delta$, we have 
\begin{equation}\label{eq:splitintcontra}
    \bigg|\int_S u~dx\bigg|> \frac{1}{N}\int_{P}u~dx.
\end{equation}
Consider now $R:=P \setminus Q$. For any $p \in P$, let $r_p = \inf \{q: |R \cap [p,q]|=\delta\}$, and let the auxiliary function $\phi$ be defined as
\begin{equation*}
    \phi: P \to \R, \  \phi(p) := \int_{R \cap [p,r_p]}u~dx.
\end{equation*}
For ease of writing, let $R_p := R \cap [p,r_p]$, and note that as $|R_p| = \delta$, by \eqref{eq:splitintcontra}, we get $\phi(p) \neq 0$ for all $p \in P$. It is easy to see that $\phi$ is continuous, as for $p,q\in P$, we have
\begin{equation*}
    |\phi(p)-\phi(q)| = \bigg|\int_{R_p}u~dx - \int_{R_q}u~dx\bigg| = \bigg|\int_{R_p \Delta R_q}u~dx\bigg|\\
    \leq \|u\|_{\infty}|R_p \Delta R_q| \leq \|u\|_{\infty}|p-q|.
\end{equation*}
Therefore, as $\phi \neq 0$, it is either strictly positive or strictly negative. Without loss of generality, we assume that $\phi(p) > 0$ for all $p$, and to avoid violating \eqref{eq:splitintcontra}, it must also be that 
\begin{equation}\label{eq:piecesbnd}
    \phi(p) > \frac{1}{N}\int_{P}u~dx
\end{equation}
for all $p \in P$. Consider $\{p_i\}_{i=1}^M \subset R$, where $p_1=0$, $p_{i+1}=r_{p_i}$ for $1 < i < M$, and $M= \left\lceil \frac{|P|}{\delta}\right\rceil-1$. Then, due to the positivity of $w$ and \eqref{eq:piecesbnd}, we have that 
\begin{equation}\label{eq:measurelemmacontra}
    \int_P u~dx= \sum_{i=1}^M \phi(p_i) + \int_Q u~dx
    > \frac{M}{N}\int_P u~dx.
\end{equation}
This will lead to a contradiction, as we will show that $M \geq N$, i.e. $\left\lceil \frac{|P|}{\delta}\right\rceil-1 \geq \left\lceil \frac{|P|}{\beta}\right\rceil$. We note that as $\delta < \beta$, it must be the case that $M \geq N-1$. Therefore, for \eqref{eq:measurelemmacontra} to avoid a contradiction, $M$ must be equal to $N-1$. This forces $\delta$ within a tight range of values: Specifically,
${|P|}/{\left\lceil \frac{|P|}{\beta}\right\rceil} \leq \delta < \beta.$
However, the exact value of $\delta$ is $|P|-\beta(\left\lceil \frac{|P|}{\beta}\right\rceil-1)$, and thus if $M = N-1$, we have 
\begin{eqnarray*}
\left\{
\begin{array}{l}
|P|-\beta\left(\left\lceil \frac{|P|}{\beta}\right\rceil-1\right) \geq \frac{|P|}{\left\lceil \frac{|P|}{\beta}\right\rceil}, \mbox{ which implies, }    |P| \geq \beta\left\lceil \frac{|P|}{\beta}\right\rceil\\
|P|-\beta\left(\left\lceil \frac{|P|}{\beta}\right\rceil-1\right)<\beta, \mbox{ or equivalently, } |P| < \beta\left\lceil \frac{|P|}{\beta}\right\rceil
\end{array}\right.
\end{eqnarray*}
This is a contradiction, so $M \geq N$, making \eqref{eq:measurelemmacontra} a contradiction. Thus the lemma holds true and such a partition of $P$ must exist. \qedhere
\end{proof}
\begin{lemma*} (Lemma~\ref{lem:gnrlpigeon})
Let $f \in L^1([0,1]^2)$, and let $S,S' \subseteq [0,1]$ be measurable subsets such that $|S|=|S'|$. Suppose for a constant $C > 0$ we have
\begin{equation*}
    \iint_{S \times S'} f~dxdy \geq C.
\end{equation*}
Then, for every $\alpha \in (0,1)$, there exist measurable sets $T\subset S$ and $T'\subset S'$ such that $|T| = |T'|=\alpha|S|$ and
\begin{equation*}
    \frac{1}{|T\times T'|}\iint_{T \times T'}f~dxdy \geq \frac{C}{|S \times S'|}.
\end{equation*}
\end{lemma*}
\begin{proof}[Proof of Lemma \ref{lem:gnrlpigeon}]
Suppose there exist integers $n,k,l$, with $l<k$, so that $|S| = \frac{k}{n}$ and $\alpha |S| = \frac{l}{n}$; the case where one or both of $|S|$ or $\alpha$ is not rational can be done using standard density/approximation arguments. 
Next, split $S$ into $k$ consecutive sets $S_1 \leq S_2 \leq \ldots \leq S_k$ of measure $\frac{1}{n}$; likewise, split $S'$ into $k$ consecutive sets $S'_1 \leq S'_2 \leq \ldots \leq S'_k$ also of measure $\frac{1}{n}$ and note that
\begin{equation*}
    \iint_{S \times S'} f~dxdy = \sum_{i = 1}^k \sum_{j=1}^k\iint_{S_i \times S'_j}f~dxdy.
\end{equation*}
Let $B_{i,j}:=\iint_{S_i \times S'_j} f~dxdy$ and let $\mathcal{N}_{k,l}$ denote the collection of all $l$-subsets of the set $\{1,...,k\}$. 
Note that there are $\binom{k-1}{l-1}$ many $l$-subsets of $k$ elements containing a specific element $i_0$. Using this, counting the number of times a specific $B_{i,j}$ appears in the following sum results in
\begin{equation*}
    \sum_{I \in \mathcal{N}_{k,l}}  \sum_{J \in \mathcal{N}_{k,l}} \sum_{\substack{i\in I\\j\in J}} B_{i,j} = 
    \sum_{i,j=1}^k \binom{k-1}{l-1}^2B_{i,j} \geq \binom{k-1}{l-1}^2C.
\end{equation*}
Thus, by the pigeonhole principle, there exist sets $I,J\in \mathcal{N}_{k,l}$ such that 
$$\sum_{\substack{i \in I\\j\in J}}B_{i,j} \geq \frac{\binom{k-1}{l-1}^2}{\binom{k}{l}^2}C = \frac{l^2}{k^2}C.$$
This implies that the sets $T = \cup_{i \in I}S_i$ and $T' = \cup_{i \in J}S'_i$ satisfy both that
\begin{equation*}
    \iint_{T \times T'} f~dxdy \geq \frac{l^2}{k^2}C
\end{equation*}
and $|T|=|T'|=\alpha|S|$, showing the lemma holds true.
\end{proof}
\end{appendices}


\end{document}